\documentclass{amsart}
\title{Stable Centres II: Finite Classical Groups}
\author{Arun S. Kannan}
\address{Department of Mathematics, Massachusetts Institute of Technology, Cambridge, MA 02139, USA}
\email{akannan@mit.edu}

\author{Christopher Ryba}
\address{Department of Mathematics, University of California, Berkeley, CA 94720, USA}
\email{ryba@math.berkeley.edu}
\usepackage{hyperref}
\usepackage{fullpage}
\usepackage{amsmath}
\usepackage{mathtools}
\usepackage{bm}
\usepackage[bibliography=common]{apxproof}
\usepackage{amsfonts}
\usepackage{amssymb}
\usepackage{ytableau}
\usepackage{subfig}
\usepackage{graphicx, caption}
\usepackage{todonotes}
\usepackage[sorting=nyt,style=alphabetic,backend=bibtex,hyperref=true,doi=false,maxbibnames=9,maxcitenames=4,url=false]{biblatex}
\DeclareRobustCommand{\qbinom}{\genfrac{\lbrack}{\rbrack}{0pt}{}}
\DeclareMathOperator{\Id}{Id}
\DeclareMathOperator{\codim}{codim}
\newcommand{\bs}{\boldsymbol}

\begin{document}
\maketitle
\begin{abstract}
Farahat and Higman constructed an algebra $\mathrm{FH}$ interpolating the centres of symmetric group algebras $Z(\mathbb{Z}S_n)$ by proving that the structure constants in these rings are ``polynomial in $n$''. Inspired by a construction of $\mathrm{FH}$ due to Ivanov and Kerov, we prove for $G_n = GL_n, U_n, Sp_{2n}, O_n$, that the structure constants of $Z(\mathbb{Z}G_n(\mathbb{F}_q))$ are ``polynomial in $q^n$'', allowing us to construct an equivalent of the Farahat-Higman algebra in each case.
\end{abstract}
\newtheorem{theorem}{Theorem}[section]
\newtheorem{lemma}[theorem]{Lemma}
\newtheorem{proposition}[theorem]{Proposition}
\newtheoremrep{corollary}[theorem]{Corollary}

\newtheorem{definition}[theorem]{Definition}
\newtheorem{example}[theorem]{Example}
\newtheorem{remark}[theorem]{Remark}

\section{Introduction}
\noindent
Let $G_n = GL_n, U_n, Sp_{2n}, O_n$ be one of the families of general linear, unitary, symplectic, and orthogonal groups. In this paper we consider the centres of the group algebras of $G_n(\mathbb{F}_q)$. Our main results are certain ``stability'' properties for the multiplication in $Z(\mathbb{Z}G_n(\mathbb{F}_q))$ which allow us to define a ``universal'' algebra $\mathrm{FH}_q^{G}$ interpolating these centres of group algebras in the parameter $n$.
\newline \newline \noindent
To illustrate the nature of these results, let us first consider the case of symmetric groups. The centre $Z(\mathbb{Z}S_n)$ has a basis consisting of conjugacy-class sums. Suppose that $g \in S_n$ has cycle type $\mu = (\mu_1, \mu_2, \ldots)$. The \emph{reduced cycle type} of $g$ is the partition $(\mu_1-1, \mu_2-1, \ldots)$ obtained by deleting the first column of the Young diagram of $\mu$. For any partition $\mu$, let $X_{\mu}$ denote the sum of all cycles of reduced cycle type $\mu$ in $S_{n}$ (which is zero if there are no such elements); the nonzero $X_\mu$ form a basis of $Z(\mathbb{Z}S_n)$. For example, $X_{(1)}$ is the sum of all transpositions. Then the square of the sum of all transpositions in $S_n$ is
\[
X_{(1)}^2 = 2 X_{(1,1)} + 3 X_{(2)} + {n \choose 2} X_{\varnothing},
\]
which is valid for any $n \geq 0$. The key feature of the above equation is that the coefficients are polynomials in $n$. Farahat and Higman \cite{FarahatHigman} proved in general that $X_\mu X_\nu$ is a linear combination of $X_{\lambda}$ with coefficients in $\mathcal{R}$, the ring of integer-valued polynomials (which has a $\mathbb{Z}$-basis consisting of binomial coefficients). They constructed a $\mathcal{R}$-algebra, $\mathrm{FH}$, with an $\mathcal{R}$-basis given by symbols $K_\mu$, and multiplicative structure constants given by the polynomials we have just described. By construction, there are surjective homomorphisms $\mathrm{FH} \to Z(\mathbb{Z}S_n)$ given by evaluating the polynomials in $\mathcal{R}$ at $n$ and sending $K_\mu$ to $X_\mu$. Farahat and Higman used this to study the modular representation theory of symmetric groups. The algebra $\mathrm{FH}$ turns out to be isomorphic to $\mathcal{R} \otimes \Lambda$, where $\Lambda$ is the ring of symmetric functions, which leads to a theory of ``content evaluation'' character formulae for the symmetric group \cite{CorteelGoupilSchaeffer}. Wang \cite{Wang2002} defined an analogous version of $\mathrm{FH}$ for the wreath products $G \wr S_n$, where $G$ is a fixed finite group, and used an associated graded version to study the cohomology of certain Hilbert schemes of points. We direct the reader to the prequel to the present paper, \cite{Ryba}, for more details regarding the algebra $\mathrm{FH}$.
\newline \newline \noindent
In this paper we construct an analogous version of $\mathrm{FH}$ for classical groups $G_n$, interpolating the centres $Z(\mathbb{Z}G_n(\mathbb{F}_q))$. For the case $G_n = GL_n$, an analogue of reduced cycle type for the general linear group called \emph{modified type} was defined by Wan and Wang \cite{Wan_Wang}. If $X_{\bm\mu}$ is the sum of all elements of $GL_n(\mathbb{F}_q)$ of modified type $\bs\mu$, let the structure constants $a_{\bm\mu\bm\nu}^{\bm{\lambda}}(n)$ be defined by
\[
X_{\bm\mu} X_{\bs\nu} = \sum_{\bs\lambda} a_{\bm\mu\bm\nu}^{\bm{\lambda}}(n) X_{\bm\lambda}.
\]
Modified types have a notion of size denoted $| \cdot |$, which puts a filtration on $Z(\mathbb{Z}GL_n(\mathbb{F}_q))$. Wan and Wang show that if $\bm\mu,\bm\nu,\bm\lambda$ are modified types, the structure constants $a_{\bm\mu\bm\nu}^{\bm{\lambda}}(n)$ are nonzero only when $| \bm\lambda | \leq | \bm\mu | + | \bm\nu |$, and are independent of $n$ when equality is achieved. In particular, this enables them to construct an algebra that specialises to the associated graded algebra of $Z(\mathbb{Z}GL_n(\mathbb{F}_q))$ for each $n$. The equivalent result for $G_n = Sp_{2n}$ was obtained by \"{O}zden in \cite{OZDEN2021263}. We prove that for any $\bm\mu,\bm\nu,\bm\lambda$ the structure constants $a_{\bm\mu\bm\nu}^{\bm{\lambda}}(n)$ are polynomials in $q^n$ (and provide a bound on the degree), thus generalising the results mentioned above. This allows us to define an algebra $\mathrm{FH}_q^{G}$ with a basis $K_{\bm\mu}$ indexed by modified types and whose coefficients lie in the \emph{ring of quantum integer-valued polynomials}, $\mathcal{R}_q$, defined by Harman and Hopkins \cite{HarmanHopkins}. This comes equipped with ``specialisation'' maps $\mathrm{FH}_q^{G} \to Z(\mathbb{Z}G_n(\mathbb{F}_q))$, analogous the case of $\mathrm{FH}$ for symmetric groups.
\newline \newline \noindent
In \cite{ivanovKerov2001symmetric}, Ivanov and Kerov construct the algebra $\mathrm{FH}$ in a very elegant way using a construction they call ``partial permutations''. These are pairs $(g, I)$ where $g \in S_\infty = \varinjlim_n S_n$ and $I$ is a finite subset of $\mathbb{Z}_{> 0}$ such that if $i \notin I$, then $i$ is a fixed point of $g$. Thus $I$ controls which symmetric groups $S_n \subseteq S_\infty$ can contain $g$. Inspired by the approach of Ivanov and Kerov, we define \emph{bounding pairs} (\emph{bounding triples} in the $GL_n$ case). These are pairs $(g, V)$ where $g \in G_\infty(\mathbb{F}_q) = \varinjlim_n G_n(\mathbb{F}_q)$ and $V$ is a finite codimension subspace of the natural representation $\mathbb{F}_q^\infty$ of $G_\infty$ ($\mathbb{F}_{q^2}^\infty$ in the case of unitary groups) on which $g$ acts as the identity. Analogously, $V$ controls which $G_n(\mathbb{F}_q) \subseteq G_\infty(\mathbb{F}_q)$ can contain $g$. A similar idea is considered for general linear groups in \cite{Meliot_GLnFq}, whose Theorem 3.7\footnote{Issues have been raised with the proof which we do not know how to address; see Remark 3.6 of \cite{Wan_Wang}.} is very similar to our Theorem \ref{gl_structure_constant_theorem}.
\newline \newline \noindent
The structure of the paper is as follows. In Section 2, we review some basic facts about quantum integer-valued polynomials and conjugacy classes in the general linear group $GL_n(\mathbb{F}_q)$. In Section 3, we recall the definitions of the classical groups and prove some basic facts about them and their conjugacy classes; a reader who is only interested in the $GL_n$ case may skip this section. In Section 4, we consider the general linear group and introduce bounding triples, which serve as an analogue to the partial permutations introduced in \cite{ivanovKerov2001symmetric}. We show that bounding triples form an algebra with an action of $GL_\infty(\mathbb{F}_q)$, and hence obtain an analogue of the construction of Ivanov and Kerov. In Section 5, we show that there are specialisation morphisms from the general linear Ivanov-Kerov algebra that surject on to $Z(\mathbb{Z}GL_n(\mathbb{F}_q))$. In Section 6, we explicitly construct the general linear Farahat-Higman algebra $\mathrm{FH}_q^{GL}$. In Section 7, we introduce bounding pairs, which are adapted for classical groups. Finally in Section 8, we construct the Farahat-Higman algebras $\mathrm{FH}_q^{G}$ for unitary, symplectic, and orthogonal groups.
\newline \newline \noindent
Note that this paper is not the paper mentioned in \cite{Ryba} that will address Iwahori-Hecke algebras of type $A$ and connections to $GL_n(\mathbb{F}_q)$. That paper will appear in due course.
\newline \newline \noindent
\textbf{Acknowledgements.} The first author would like to thank his undergraduate advisor, Weiqiang Wang, for introducing him to this problem. Both authors would like to thank Pavel Etingof for helpful discussions. This paper is based upon work supported by The National Science Foundation Graduate Research Fellowship Program under Grant No. 1842490 awarded to the first author.

\section{Background}
\subsection{Quantum Integer-Valued Polynomials} \label{qivp_subsection}
Let $q$ be a formal variable, which later on we will often specialize to be a power of a prime number. For any integer $n$, we let $[n]_q$ denote the $q$-integer
\[
[n]_q = 1 + q + \cdots + q^{n-1} = \frac{q^n - 1}{q - 1}.
\] 
We can then define the $q$-factorial $[n]_q!$ of a $q$-integer $[n]_q$. Let $[0]_q! = 1$ and define
\[
[n]_q! = [n]_q[n-1]_q \cdots [1]_q = \frac{(q^n-1)(q^{n-1} - 1)\cdots (q-1)}{(q-1)^n}.
\]
Finally, we can define the $q$-binomial coefficient, or Gaussian binomial coefficient, by
\[
\qbinom{n}{k}_q = \frac{[n]_q!}{[k]_q![n-k]_q!}.
\]
for $0 \leq k \leq n$, which turns out to be an element of $\mathbb{Z}[q,q^{-1}]$ (in fact of $\mathbb{Z}_{\geq 0}[q]$). Notice that at $q = 1$, we recover the usual definitions of an integer, factorial, and binomial coefficient. These $q$-integers are useful for describing various quantities associated to finite fields.
Let us briefly set $q$ to be the size of the finite field $\mathbb{F}_q$. It is well known that
\[
|GL_n(\mathbb{F}_q)| = (q^n-1)(q^n-q)\cdots (q^{n} - q^{n-1}) =  q^{\frac{n(n-1)}{2}}(q-1)^n[n]_q!.
\]
Similarly, the number of $k$-dimensional subspaces of $\mathbb{F}_q^n$ is $\qbinom{n}{k}_{q}$.
\newline \newline \noindent
Returning to the case where $q$ is a formal variable, we now discuss the ring of \emph{quantum integer valued polynomials} introduced by Harman and Hopkins in \cite{HarmanHopkins}.
\begin{definition}[Section 4 \cite{HarmanHopkins}] \label{qivp_definition}
Let $\mathcal{R}_q$ be the set of polynomials $f(x) \in \mathbb{Q}(q)[x]$ such that
\[
f([n]_q) \in \mathbb{Z}[q,q^{-1}]
\]
for all $n \in \mathbb{Z}$.
\end{definition}
\noindent
It is not difficult to see that $\mathcal{R}_q$ is a ring (in fact a $\mathbb{Z}[q,q^{-1}]$-algebra). Much like the usual ring of integer-valued polynomials, this ring admits a concise explicit description.
\begin{theorem}[Propositions 1.2, 4.3, \cite{HarmanHopkins}] \label{qivp_structure_theorem}
As a $\mathbb{Z}[q,q^{-1}]$-module, $\mathcal{R}_q$ is free with basis
\[
\qbinom{x}{k} = \frac{x(x-[1]_q) \cdots(x-[k-1]_q)}{q^{k \choose 2} [k]_q!}
\]
where $k \in \mathbb{Z}_{\geq 0}$. (Here $\qbinom{x}{0} = 1$.)
\end{theorem}
\noindent
These basis elements obey
\[
\qbinom{[n]_q}{k} = \qbinom{n}{k}_q,
\]
and this is the main reason why $\mathcal{R}_q$ will be important for us. We will require a ring that includes $q^{-1}$ and also $\qbinom{n}{k}_q$ where $n$ is viewed as a variable. We achieve this by working with $\qbinom{x}{k}$ and evaluating at $x = [n]_q$ where necessary. Performing the change of variables $t = 1 + (q-1)x$, we see that $x = [n]_q$ is equivalent to $t = q^n$, and 
\[
\qbinom{x}{k} = \frac{x(x-[1]_q) \cdots(x-[k-1]_q)}{q^{k \choose 2} [k]_q!} = \frac{(t-1)(t-q) \cdots (t-q^{k-1})}{(q^k -1) (q^k - q) \cdots (q^k - q^{k-1})}.
\]
If we express $\mathcal{R}_q$ in terms of the variable $t$ instead of $x$, we get the ring of polynomials $f(t) \in \mathbb{Q}(q)[t]$ such that $f(q^n) \in \mathbb{Z}[q,q^{-1}]$. 
\newline \newline \noindent
The original definition of $\mathcal{R}_q$ using $\qbinom{x}{k}$ is better behaved at $q=1$ (our change of variables is singular at $q=1$). Working with the variable $t$ allows us to phrase statements in terms of evaluating at $q^n$, which may be slightly more transparent than evaluating at $[n]_q$. Ultimately we set $q$ to be the size of a finite field (obtaining a subring of $\mathbb{Q}[x]$ or $\mathbb{Q}[t]$), and the reader may choose whichever incarnation of $\mathcal{R}_q$ they prefer.

\begin{lemma} \label{shifted_qivp_lemma}
Let $d \in \mathbb{Z}$ and $h \in \mathbb{Z}_{\geq 0}$. Then there is an element $f(x) \in \mathcal{R}$ such that
\[
f([m]_q) = \qbinom{m+d}{h}_q
\]
\end{lemma}
\begin{proof}
We have
\[
\qbinom{m+d}{h}_q = \qbinom{[m+d]_q}{h} = \qbinom{q^{d}[m]_q + [d]_q}{h}
\]
is an element of $\mathbb{Z}[q,q^{-1}]$ but also the evaluation of $\qbinom{q^{d}x + [d]_q}{h} \in \mathbb{Q}(q)[x]$ at $x = [m]_q$. So $\qbinom{q^{d}x + [d]_q}{h}$ satisfies Definition \ref{qivp_definition} and is hence an element of $\mathcal{R}_q$. We conclude that $\qbinom{m+d}{h}_q$ is obtained by evaluating an element of $\mathcal{R}_q$ at $[m]_q$. 
\end{proof}

\subsection{Multipartitions and Conjugacy Classes in \texorpdfstring{$GL_n$}{General Linear Groups}}
A \emph{partition} $\lambda$ is a non-increasing sequence in $\mathbb{Z}_{\geq 0}$, $\lambda = (\lambda_1, \lambda_2, \ldots)$ of which only finitely many terms are nonzero. We write $\mathcal{P}$ for the set of partitions. The numbers $\lambda_1, \lambda_2, \ldots$ are called the $\emph{parts}$ of $\lambda$. An alternative way of describing a partition is to specify the number of times a given element of $\mathbb{Z}_{>0}$ appears. So if $\lambda$ has $m_i$ parts of size $i$, we write $\lambda = (1^{m_1} 2^{m_2} \cdots )$. Similarly, we write $m_i(\lambda)$ to mean the number of times $i$ appears as a part of $\lambda$. The \emph{length} of $\lambda$, $l(\lambda)$, is the number of nonzero parts:
\[
l(\lambda) = \sum_{i \in \mathbb{Z}_{>0}} m_i(\lambda).
\]
The \emph{size} of $\lambda$, $|\lambda|$, is the sum of the parts:
\[
|\lambda| = \sum_{i \geq 1} \lambda_i.
\]
One statistic of partitions that turns out to be important is the following.
\begin{definition} \label{partition_n_definition}
If $\lambda \in \mathcal{P}$, let
\[
\mathbf{n}(\lambda) = \sum_{i \geq 1} (i-1) \lambda_i.
\]
\end{definition}
\noindent
We use boldface to avoid confusion with $n$, which will typically denote the size of a matrix.
\newline \newline \noindent
If $k$ is any field, conjugacy classes in $GL_n(k)$ are determined by primary rational canonical form. To an $n \times n$ matrix $M$, we associate a function $\bs\mu$, called a \emph{multipartition}, from the set of monic irreducible polynomials over $k$ to $\mathcal{P}$, such that $m_i(\bs\mu(r))$ is the multiplicity of $r(t)^i$ as an elementary divisor of the $k[t]$-module $k^n$, where $t$ acts by the matrix $M$. We refer to $\bs\mu$ as the \emph{type} of $M$. Thus two matrices are similar if and only if their corresponding types are equal.
\noindent
From now on, we will only consider invertible matrices. This amounts to excluding the irreducible polynomial $t$. Let $\Phi_q$ be the set of monic irreducible polynomials in $\mathbb{F}_q[t]$ that are different from $t$. Thus the types of matrices in $GL_n(\mathbb{F}_q)$ are functions $\bs\mu: \Phi_q \to \mathcal{P}$ such that
\[
\sum_{r \in \Phi_q} \deg(r) |\bs\mu(r)| = n.
\]
We refer to the above sum as the size of $\bs\mu$ and write $\bs\mu \vdash n$ to indicate this. We will sometimes write $\bs\mu \cup (1^d)_{t-1}$ to indicate a multipartition obtained from $\bs \mu$ by appending $d$ parts of size 1 to $\bs\mu(t-1)$. Concretely, if $g \in GL_n(\mathbb{F}_q)$ has type $\bs\mu$, then  $\bs\mu \cup (1^d)_{t-1}$ is the type of the image of $g$ under the embedding into $GL_{n+d}(\mathbb{F}_q)$ by adding $1$-s along the main diagonal.

\begin{example}\label{basic_example}
Consider the following matrix $g \in GL_{11}(\mathbb{Q})$:

\[ 
\setcounter{MaxMatrixCols}{11}
\begin{bmatrix}
1 & 1 & 0 & & & & & & & & \\
0 & 1 & 1 & & & & & & & & \\
0 & 0 & 1 & & & & & & & & \\
& & & 1& & & & & & & \\
& & & & 1& & & & & & \\
& & & & & 0 & 0 & 0 & -1 & & \\
& & & & & 1 & 0 & 0 & 0 & & \\
& & & & & 0 & 1 & 0 & -2 & & \\
& & & & & 0 & 0 & 1 & 0 & &\\
& & & & & & & & &  0& -1  \\
& & & & & & & & &  1 & 0  
\end{bmatrix},
\]
where empty entries are assumed to be zero. Then, the type $\bs\mu$ of $g$ is given by $\bs\mu(t-1) = (3,1,1)$, $\bs\mu(t^2 + 1) = (2,1)$, and $\bs\mu(f) = \varnothing$ for all other irreducible polynomials $f$ over $\mathbb{Q}$ (the fourth block in the matrix above is the companion matrix of $(t^2 + 1)^2$). Notice that $\deg(t-1)|(3,1,1)| + \deg(t^2+1)|(2,1)| = (1)(5) + (2)(3) = 11$, which confirms $\bs\mu \vdash 11$.
\end{example}
\noindent
There is a formula for the sizes of centralisers in $GL_n(\mathbb{F}_q)$.
\begin{proposition}[Section 4.2 \cite{Macdonald1995}] \label{gl_centraliser_prop}
If $M \in GL_n(\mathbb{F}_q)$ has type $\bs\mu$, then the size of the centraliser of $M$ is
\[
\prod_{r \in \Phi_q} q_r^{|\bs\mu(r)| + 2\mathbf{n}(\bs\mu(r))} \prod_{i \geq 1} \varphi_{m_i(\bs\mu(r))}(q_r^{-1}),
\]
where $q_r = q^{\deg(r)}$, $\varphi_k(t) = \prod_{i=1}^k (1-t^i)$, and $\mathbf{n}(\bs\mu(r))$ is the function defined in Definition \ref{partition_n_definition}.
\end{proposition}
\noindent
We will only need this formula for the following result.
\begin{corollary} \label{gln_centraliser_corollary}
If $g \in GL_n(\mathbb{F}_q)$. Then the block matrix $\tilde{g} = \bigl( \begin{smallmatrix}g & 0 \\ 0 & \Id_d \end{smallmatrix}\bigr)$ may be viewed as an element of $GL_{n+d}(\mathbb{F}_q)$. The ratio of the sizes of the centralisers of these elements is
\[
\frac{|C_{GL_{n+d}(\mathbb{F}_q)}(\tilde{g})|}{|C_{GL_n(\mathbb{F}_q)}(g)|} = q^{d(2k +d)} \prod_{i=h+1}^{h+d} (1 - q^{-i}),
\]
where $\bs\mu$ is the type of $g$, $k=l(\bs\mu(t-1)$ and $h=m_1(\bs\mu(t-1))$.
\end{corollary}
\begin{proof}
The type of $\tilde{g}$ is $\bs\mu \cup (1^d)_{t-1}$. Applying Proposition \ref{gl_centraliser_prop}, for both $g$ and $\tilde{g}$, we find that the factors associated to $r \neq t-1$ are equal for both elements and cancel out in the fraction. Similarly $\varphi_{m_i(\bs\mu(t-1))}$ is unchanged for $i \geq 2$ and cancels out. Appending $d$ parts of size $1$ to $\bs\mu(t-1)$ increases $\mathbf{n}(\bs\mu(t-1))$ by $(k) + (k+1) + \cdots + (k+d-1) = kd + d(d-1)/2$. So we are left with
\[
q^{d + 2kd + d(d-1)} \frac{\varphi_{h+d}(q^{-1})}{\varphi_h(q^{-1})}
=
q^{d(2k + d)} \prod_{i=h+1}^{h+d} (1 - q^{-i}).
\]
\end{proof}

\section{Finite Classical Groups}\label{finite_classical_groups_section}
\noindent
In this section, we recall some basic properties about the finite classical groups, by which we mean to include unitary, symplectic, and orthogonal groups over a finite field. The main reference for this section is the paper \cite{Wall}, although Peter Cameron's notes \cite{CameronNotes} are also helpful. The classical groups are each defined as a symmetry group preserving a non-degenerate sesquilinear form. Up to a point, these may be treated simultaneously. However, in characteristic 2, orthgonal groups must be treated using quadratic forms instead, so whenever dealing with orthogonal groups or symmetric bilinear forms, we assume the characteristic is different from 2.
\newline \newline \noindent
We fix an automorphism $\sigma$ of the ground field $k$. For the symplectic and orthogonal groups, $\sigma$ will be the identity map on $k = \mathbb{F}_q$. For the unitary groups, we take  $k=\mathbb{F}_{q^2}$ and $\sigma$ will be the nontrivial element of $\mathrm{Gal}(\mathbb{F}_{q^2}/\mathbb{F}_q)$, which may be written as $x \mapsto x^q$. A $\sigma$-\emph{sesquilinear form} on a vector space $V$ is a map $B: V \times V \to k$ which is linear in the first argument, and $\sigma$-linear in the second argument, i.e. for $x_1, x_2 \in k$ and $v_1, v_2, v_3 \in V$,
\begin{eqnarray*}
B(x_1 v_1 + x_2 v_2, v_3) &=& x_1 B(v_1, v_3) + x_2 B(v_2, v_3), \\
B(v_1, x_1 v_2 + x_2 v_3) &=& \sigma(x_1) B(v_1, v_2) + \sigma(x_2) B(v_1, v_3).
\end{eqnarray*}
In the case where $\sigma$ is the identity map, this is simply a bilinear form. If $B(v,w) = \sigma(B(w,v))$, then $B$ is said to be $\sigma$-\emph{Hermitian}. Since a nonzero bilinear form takes every value in the ground field, the equation $B(v,w) = \sigma^2(B(v,w))$ shows that nontrivial cases only arise when $\sigma^2 = 1$. We shall at times refer to a $\sigma$-sesquilinear form as simply a sesquilinear form.
\newline \newline \noindent
A bilinear form $B$ is said to be \emph{reflexive} provided that $B(v,w) = 0$ if and only if $B(w,v) = 0$. This condition guarantees that if $U$ is a subspace of $V$, then the left and right orthogonal spaces to $U$ coincide. In particular, when $U$ is the trivial subspace, we see that the left and right radicals of $B$ coincide:
\[
\{ v \in V | B(v,w) = 0, \hspace{5mm} \forall w \in V\} = \{w \in V | B(v,w) = 0, \hspace{5mm} \forall v \in V\}.
\]
It is a standard fact that a reflexive bilinear $B$ form must be either symmetric (i.e. $B(v,w) = B(w,v)$) or alternating ($B(v,v) = 0$). Similarly, a reflexive $\sigma$-sesquilinear form is either alternating or a scalar multiple of a $\sigma$-Hermitian form. So we do not lose too much generality by restricting ourselves to symmetric, alternating, or Hermitian forms (each of which is automatically reflexive). Henceforth, we assume that we are in one of the following three cases:
\begin{enumerate}
\item $V$ is a vector space over $\mathbb{F}_{q^2}$ and $B$ is a $\sigma$-Hermitian form where $\sigma$ is the nontrivial element of $\mathrm{Gal}(\mathbb{F}_{q^2}/\mathbb{F}_q)$.
\item $V$ is a vector space over $\mathbb{F}_q$ and $B$ is an alternating bilinear form,
\item $V$ is a vector space over $\mathbb{F}_q$ which has characteristic different from 2, and $B$ is a symmetric bilinear form,
\end{enumerate}
\noindent
In the latter two cases we will sometimes write $\sigma$ for the identity map in order to treat all three cases simultaneously. So $\sigma$ will be an automorphism of the ground field obeying $\sigma^2 = \Id$.
\begin{remark}
The ``alternating $\sigma$-Hermitian'' constraint $\sigma(B(w,v)) = -B(v,w)$ is compatible with our setup, but not meaningfully different from the $\sigma$-Hermitian case. This is immediate in characteristic 2, so assume that we are in odd characteristic. Note that there is a solution to $x^{q-1} = -1$ in $\mathbb{F}_{q^2}$ because $\mathbb{F}_{q^2}^\times$ is cyclic of order $q^2-1$, so we may take $x$ to be any element of order $2(q-1)$. Then since $\sigma(x) = x^q = -x$, we have $\sigma(B(w,v)) = -B(v,w)$ if and only if $\sigma(xB(w,v)) = xB(v,w)$. So we may interconvert between ``alternating $\sigma$-Hermitian'' and $\sigma$-Hermitian forms by multiplying the forms by the scalar $x$.
\end{remark}
\noindent
Suppose that the vector space $V$ has two sesquilinear forms $B_1$ and $B_2$. We say that these forms are equivalent if there is $g \in GL(V)$ such that $B_1(gv, gw) = B_2(v,w)$.
\begin{definition}
Let $B: V \times V \to k$ be a non-degenerate sesquilinear form. The symmetry group of $B$ is
\[
G_B(V) = \{g \in GL(V) \mid B(gv,gw) = B(v,w), \hspace{2mm} v,w \in V\}.
\]
\end{definition}
\noindent
Concretely, we may choose a basis of $V$ and let $M_B$ be the matrix associated to the bilinear form $B$, so that $B(v,w) = v^T M_B \sigma(w)$. Then the forms $B_1$ and $B_2$ are equivalent when there is a matrix $g$ such that $g^T M_{B_1} \sigma(g) = M_{B_2}$. Similarly, the symmetry group $G_B(V)$ consists of matrices $g$ such that $g^T M_B \sigma(g) = M_B$. The orthogonal, symplectic, and unitary groups are the symmetry groups of non-degenerate symmetric, alternating, and Hermitian forms, respectively. However, in general, two forms $B_1$ and $B_2$ on $V$ of the same kind might not be equivalent. In particular, it might happen that $G_{B_1}(V)$ and $G_{B_2}(V)$ are not isomorphic, so the term orthogonal/symplectic/unitary group could be ambiguous if the form $B$ is not specified. This leads us to the next topic: a classification of such forms over finite fields.
\begin{definition}
Suppose that $B$ is a sesquilinear form on $V$. A \emph{hyperbolic plane} for $V$ is a two dimensional subspace of $V$ spanned by vectors $u_1,u_2$ such that $B(u_1,u_1) = B(u_2,u_2)= 0$ and $B(u_1,u_2) = 1$.
\end{definition}
\noindent
Of course, the value of $B(u_2,u_1)$ is determined by whether the form is symmetric, alternating, or Hermitian. The utility of hyperbolic planes rests on the following result.
\begin{lemma} \label{hyperbolic_plane_splitting_lemma}
Let $B$ be a non-degenerate sesquilinear form on $V$. Suppose that $u_1,u_2$ span a hyperbolic plane $U$ in $V$. Then $V = U \oplus U^\perp$.
\end{lemma}
\noindent
Before we elaborate on how hyperbolic planes allow us to classify sesquilinear forms, we record some technical results.

\begin{lemma} \label{trace_value_lemma}
Suppose that $B$ is a $\sigma$-Hermitian form on $V$, where $\sigma$ is a field automorphism of $k$ of order $2$. Then for any $v \in V$ there exists $x \in k$ such that $B(v,v) = x + \sigma(x)$.
\end{lemma}
\begin{proof}
Let $k_0$ be the subfield of $k$ fixed pointwise by $\sigma$, and let $T$ denote the set of elements in $k$ of the form $x + \sigma(x)$. Because $\sigma$ is an involution, $T$ is contained in $k_0$. Now, notice that $T$ is closed under addition as well as by multiplication by elements of $k_0$. If we view $k_0$ as a one-dimensional vector space over itself, this means that $T$ is a subspace of $k_0$. Therefore, either $T=k_0$ or $T = 0$. If $T = 0$, then $x + \sigma(x) = 0$ for all $x$. But $\sigma(x) = -x$ only defines an automorphism if the characteristic of $k$ is 2, in which case $\sigma$ is the identity map, contradicting the assumption on $\sigma$. Therefore, $T = k_0$. By the $\sigma$-Hermitian property, $\sigma(B(v,v)) = B(v,v)$, so that $B(v,v) \in k_0 = T$.
\end{proof}
\begin{proposition} \label{hyperbolic_plane_splitting_prop}
Let $B$ be a nondegenerate sesquilinear form on a vector space $V$. Suppose that there is $u_1 \in V$ such that $B(u_1, u_1) = 0$. Then there exists $u_2 \in V$ such that $u_1, u_2$ span a hyperbolic plane.
\end{proposition}
\begin{proof}
Since $B$ is non-degenerate, the must exist some $v \in V$ with $B(u_1,v) \neq 0$, which we rescale to assume $B(u_1, v) = 1$. Then
\begin{itemize}
\item if $B$ is Hermitian, let $x \in k$ be such that $B(v,v) = x + \sigma(x)$ (such $x$ is guaranteed to exist by Lemma \ref{trace_value_lemma}). Then $B(v-xu_1, v-xu_1) = 0$, so we may take $u_2 = v - xu_1$.
\item if $B$ is alternating, $B(v,v) = 0$, so we may take $u_2 = v$.
\item if $B$ is symmetric, let $x = \frac{1}{2} B(v,v)$. Then $B(v-xu_1, v-xu_1) = 0$ and we may take $u_2 = v-xu_1$.
\end{itemize}
\end{proof}

\begin{lemma} \label{lagrangian_subspace_lemma}
Suppose that $V$ is a $2n$-dimensional vector space with a non-degenerate sesquilinear form $B$, and $W$ is an $n$-dimensional subspace of $V$ such that $B$ restricted to $W$ is zero. Then there is an $n$-dimensional subspace $W^\prime$ of $V$ such that $V = W^\prime \oplus W$ and $B$ restricted to $W^\prime$ is zero.
\end{lemma}
\begin{proof}
Given any nonzero vector $w \in W$, we have $B(w, w) = 0$, so Proposition \ref{hyperbolic_plane_splitting_prop} shows there exists $w' \in V$ such that $w$ and $w'$ span a hyperbolic plane $U$. In particular, this means that $U^\perp$ is $2(n-1)$ dimensional and $w, w' \not\in U^{\perp}$ by Lemma \ref{hyperbolic_plane_splitting_lemma}, and $w' \not \in W$ because $B(w, w') = 1$. Therefore, $W \cap U^{\perp}$ is $n-1$ dimensional. The restriction of $B$ to $U^{\perp}$ is nondegenerate, and $B$ vanishes on $W \cap U^\perp$. Proceeding inductively on $U^{\perp}$ instead of $V$ and $W \cap U^{\perp}$ instead of $W$, we deduce that $V$ is a direct sum of $n$ hyperbolic planes spanned by pairs $\{w, w^\prime\}$, where the $w$ form a basis of $W$. We take $W^\prime$ to be spanned by the $w^\prime$.
\end{proof}
\noindent
The following proposition tells us when two subspaces are equivalent under the action of the symmetry group $G_B(V)$. It turns out that it is necessary and sufficient that $B$ should restrict to an equivalent sesquilinear form on the two subspaces. This result is sometimes known as Witt's Lemma.

\begin{proposition}[Theorem 1.2.1, \cite{Wall}] \label{witt_lemma_proposition}
Suppose that $V$ is equipped with a non-degenerate sesquilinear form $B$ and $W_1, W_2$ are subspaces of $V$ such that there is a bijective linear map $g: W_1 \to W_2$ preserving $B$, i.e. $B(gw, gw^\prime) = B(w,w^\prime)$ for $w,w^\prime \in W_1$. Then there exists $\tilde{g} \in G_B(V)$ such that the restriction of $\tilde{g}$ to $W_1$ is g.
\end{proposition}

\noindent
Now we resume our discussion of the classification of sesquilinear forms.
\begin{definition}
If every nonzero vector $v \in V$ obeys $B(v,v) \neq 0$, we say that $(V, B)$ is \emph{anisotropic}.
\end{definition}
\noindent
Anisotropic spaces are automatically non-degenerate. Note that a zero-dimensional space is considered to be anisotropic.
\begin{proposition} \label{sesquilinear_structure_prop}
If $V$ is a space with a sesquilinear form $B$, we may write $V = \ker(B) \oplus U^{\oplus r} \oplus W$, where $U$ is a hyperbolic plane, $W$ is anisotropic, and each summand is orthogonal to all the others under $B$.
\end{proposition}
\begin{proof}
We pick an arbitrary splitting $V = \ker(B) \oplus M$, which is automatically orthogonal. The restriction of $B$ to $M$ is non-degenerate. If $M$ is not anisotropic, there is a vector $u \in M$ such that $B(u,u) = 0$, and hence there is a hyperbolic plane $U$ contained inside $M$. Then Lemma \ref{hyperbolic_plane_splitting_lemma} shows we may split off this hyperbolic plane to get $M = U \oplus U^\perp$. Repeating this with $U^\perp$ in place of $M$, the dimension decreases until we eventually obtain an anisotropic space (possibly zero).
\end{proof}
\noindent
The decomposition $V = \ker(B) \oplus U^{\oplus r} \oplus W$ in Proposition \ref{sesquilinear_structure_prop} is not unique. However, the multiplicity $r$ and the equivalence class of $W$ (equipped with the restriction of $B$) are determined; $r$ is called the \emph{polar rank} of $V$ and $(W, B)$ is called the \emph{germ} or \emph{core} of $V$. The assertion that these are indeed invariants of $(V, B)$ is sometimes called Witt's Theorem. As a result, the classification of spaces $(V, B)$ is reduced to the classification of anisotropic forms. We merely state the result.
\begin{theorem} \label{anisotropic_space_classification_theorem}
Up to equivalence, the nonzero anisotropic spaces $(V,B)$ are as follows:
\begin{itemize}
\item if $B$ is Hermitian, the only nonzero anisotropic space is 1-dimensional (spanned by $v$, say, with form $B(v, v) = 1$).
\item if $B$ is alternating, there are no nonzero anisotropic spaces,
\item if $B$ is symmetric, there are two 1-dimensional anisotropic spaces, and one 2-dimensional space. The 1-dimensional spaces may be presented by the forms $B(x, y) = m x y$, where $m \in \mathbb{F}_q$ either is, or is not, a perfect square. A form representing the two dimensional anisotropic space is 
\[
B((x_1,x_2), (y_1, y_2)) = x_1y_1 - m x_2y_2,
\]
where $m$ is a non-square in $\mathbb{F}_q$. \end{itemize}
\end{theorem}
\noindent
A consequence of this theorem is the following proposition.
\begin{proposition}
There is a unique Hermitian form (up to equivalence) on $V$ regardless of $\dim(V)$. A non-degenerate alternating form on $V$ can only exist when $\dim(V)$ is even and any two such forms on $V$ are equivalent. There are two equivalence classes of symmetric forms on $V$ in any dimension.
\end{proposition}
\begin{proof}
A non-degenerate space $(V, B)$ is obtained from its germ by adding some number of hyperbolic planes, which means that the dimension of the germ has the same parity as $\dim(V)$. In the Hermitian case, we see that there is only one anisotropic space whose dimension has a given parity, so the decomposition of a $V$ into hyperbolic planes and the germ is determined by $\dim(V)$. We also see that any alternating form is a direct sum of hyperbolic planes. On the other hand, in the symmetric case there are two anisotropic symmetric forms whose dimension has a given parity. So adding an appropriate number of hyperbolic planes gives two equivalence classes of forms.
\end{proof}
\noindent
The salient consequence of this classification is that the symmetry groups of non-degenerate Hermitian or alternating forms of a given dimension are all isomorphic, so we may unambiguously refer to their symmetry groups without specifying the sesquilinear form they preserve. Furthermore, we may pick a particular form without loss of generality. The case of a symmetric bilinear form is more subtle.

\begin{definition}
If $k$ is any field, the \emph{Witt ring} of $k$, denoted $W(k)$ consists of equivalence classes of anisotropic spaces over $k$. The the sum of $(V_1, B_1)$ and $(V_2, B_2)$ is the germ of $(V_1 \oplus V_2, B_1 \oplus B_2)$. (There is also a multiplication given by the tensor product, but we will not need it.)
\end{definition}
\noindent
Let $q$ be an odd prime power, and consider the Witt ring $W(\mathbb{F}_q)$. Let $\mathbf{0}$ be the zero-dimensional anisotropic space, while $\mathbf{1}$ and $\delta$ are the one-dimensional anisotropic spaces corresponding to the square and non-square cases respectively. Finally, let $\omega$ be the two-dimensional anisotropic space. Then $W(\mathbb{F}_q) = \{\mathbf{0}, \mathbf{1}, \delta, \omega\}$. Here $\mathbf{0}$ is the additive identity. We will denote the sum in $W(\mathbb{F}_q)$ with $\oplus$, so for example $\omega \oplus \omega = \mathbf{0}$. (It turns out that $W(k)$ is isomorphic to $\mathbb{Z}/2\mathbb{Z} \times \mathbb{Z}/2\mathbb{Z}$ if $q$ is $1$ modulo $4$, and $\mathbb{Z}/4\mathbb{Z}$ if $q$ is $3$ modulo $4$.)

\begin{lemma}
Let $V$ be an odd-dimensional space over $\mathbb{F}_q$. Then the symmetry groups of any two non-degenerate symmetric bilinear forms on $V$ are isomorphic.
\end{lemma}
\begin{proof}
Suppose the form $B$ is represented by the matrix $M$ with respect to some choice of basis. For any $g \in GL(V)$, $\det(g^T M g) = \det(M) \det(g)^2$, so we may characterise the two equivalence classes of forms as having matrices whose determinant is or is not a square in $\mathbb{F}_q$. Multiplying the form by a non-square scalar $m$ does not change the symmetry group, but multiplies $\det(M)$ by $m^{\dim(V)}$ (which is a non-square), yielding a form in the other equivalence class.
\end{proof}
\noindent
It turns out that there is no such coincidence in the even-dimensional case; the two equivalence classes of symmetric forms have different symmetry groups when $\dim(V)$ is even.
\begin{definition}
We write $U_n(\mathbb{F}_q)$ for the symmetry group of a Hermitian bilinear form on an $n$-dimensional vector space over $\mathbb{F}_{q^2}$.
Similarly, we write $Sp_{2n}(\mathbb{F}_q)$ for the symmetry group of an alternating bilinear form on a $2n$-dimensional vector space over $\mathbb{F}_q$. It is standard to write $O_{2n+1}(\mathbb{F}_q)$ for the symmetry group of a symmetric form on a $(2n+1)$-dimensional vector space over $\mathbb{F}_q$, as well as $O_{2n}^+(\mathbb{F}_q)$ and $O_{2n}^-(\mathbb{F}_q)$ for the symmetry groups of symmetric forms on a space of dimension $2n$ over $\mathbb{F}_q$, whose germs are $\mathbf{0}$ and $\omega$, respectively. When necessary, we will also write $O_{2n+1}^{+}(\mathbb{F}_q)$ and $O_{2n+1}^{-}(\mathbb{F}_q)$ for the symmetry groups for the symmetric bilinear forms on $\mathbb{F}_q^{2n+1}$ with germs $\mathbf{1}$ and $\delta$ respectively, even though these groups are isomorphic.
\end{definition}
\noindent
In Section \ref{classical_groups_section}, we will need some numerical information about finite classical groups. We recall the sizes of the groups, as well as the sizes of centralisers of (elements of) conjugacy classes.
\begin{proposition}[Subsection 2.6, \cite{Wall}] \label{classical_group_sizes_proposition}
We have
\begin{eqnarray*}
|U_n(\mathbb{F}_q)| &=& q^{n \choose 2} \prod_{i=1}^n (q^i - (-1)^i), \\
|Sp_{2n}(\mathbb{F}_q) &=& q^{n^2} \prod_{i=1}^n (q^{2i}-1).
\end{eqnarray*}
If $q$ is odd, then
\begin{eqnarray*}
|O_{2n+1}(\mathbb{F}_q)| &=& 2 q^{n^2} \prod_{i=1}^n (q^{2i}-1), \\
|O_{2n}^{\pm}(\mathbb{F}_q)| &=& 2 q^{n^2-n}(q^n \mp 1) \prod_{i=1}^{n-1} (q^{2i}-1).
\end{eqnarray*}
\end{proposition}
\noindent
Now we state some facts about the ratios of the sizes of certain centralisers in classical groups. These sizes were first calculated by Wall in Subsection 2.6 of \cite{Wall}, but the results are restated more transparently in \cite{Fulman_cycle_indices}. We defer the proofs of these facts to the appendix due to technical details that we will not need later in the paper. In the statements below, when we refer to the type of an element in a classical group, we mean the the type of that element when viewed as an element of the ambient general linear group.
\par

\begin{toappendix}
\noindent
Let us retain the notation from Section \ref{finite_classical_groups_section}.
\begin{definition}
If $r(t) = \sum_i a_i t^i \in \Phi_{q^2}$ is a monic irreducible polynomial with coefficients in $\mathbb{F}_{q^2}$ other than $t$, define
\[
r^*(t) = \frac{t^{\deg(r)}}{\sigma(r(0))} \sum_i \sigma(a_i)t^{-i}
\]
(where $\sigma(x) = x^q$), noting that $r^*(t) \in \Phi_{q^2}$ also.
\end{definition}
\noindent
The only example that will be relevant is that this operation sends $t-1$ to itself. It turns out that the conjugacy class of an element of $U_n(\mathbb{F}_q)$ is determined by its conjugacy class as an element of $GL_n(\mathbb{F}_{q^2})$, i.e. its type. However only types $\bs\mu$ with $\bs\mu(r) = \bs\mu(r^*)$ correspond to elements of the unitary group.

\begin{definition}
We define two functions, $A$ and $B$, that implicitly depend upon $g \in U_n(\mathbb{F}_q)$. Let $\bs\mu$ be the type of $g$. Let $A$ be the following function defined on $\Phi_{q^2} \times \mathbb{Z}_{>0}$.
\[
A(r, i) =
\left\{
        \begin{array}{ll}
            |U_{m}(\mathbb{F}_Q)| & \quad \mbox{if } r = r^* \\
            |GL_{m}(\mathbb{F}_{Q^2})|^{\frac{1}{2}} & \quad \mbox{if } r \neq r^*
        \end{array}
    \right.
\]
where $m = m_i(\bs\mu(r))$ and $Q = q^{\deg(r)}$. Now define
\[
B(r) = Q^{\sum_{i < j} 2i m_i(\bs\mu(r))m_j(\bs\mu(r)) + \sum_{i \geq 1} (i-1) m_i(\bs\mu(r))^2} \prod_{i \geq 1} A(r, i).
\]
\end{definition}
\begin{proposition}[Section 2.6, \cite{Wall}]
The centraliser of an element $g \in U_n(\mathbb{F}_q)$ has size
\[
\prod_{r \in \Phi_{q^2}} B(r).
\]
\end{proposition}
\noindent
Although $A$ and $B$ are complicated, ultimately we will only care about what happens when when $g$ is extended to a larger matrix by taking a block sum with the identity matrix of some size $d$.

\end{toappendix}
\begin{corollaryrep} \label{unitary_centraliser_ratio_corollary}
Suppose that $g \in U_n(\mathbb{F}_q)$. Then the block matrix $\tilde{g} = \bigl( \begin{smallmatrix}g & 0 \\ 0 & \Id_d \end{smallmatrix}\bigr)$ may be viewed as an element of $U_{n+d}(\mathbb{F}_q)$. The ratio of the centralisers of these elements is
\[
\frac{C_{U_{n+d}(\mathbb{F}_q)}(\tilde{g})}{C_{U_n(\mathbb{F}_q)}(g)} = q^{2d(k - h)} \frac{|U_{h + d}(\mathbb{F}_q)|}{|U_{h}(\mathbb{F}_q)|},
\]
where $\bs\mu$ is the type of $g$, $k=l(\bs\mu(t-1))$ and $h=m_1(\bs\mu(t-1))$.
\end{corollaryrep}

\begin{toappendix}

\begin{proof}
The type of $\tilde{g}$ is $\bs\mu \cup (1^{d})_{t-1}$. Computing the size of the centraliser, the factors $B(r)$ for $r \neq t-1$ are identical for $g$ and $\tilde{g}$, and so they cancel out in the fraction. As for $B(t-1)$, we have $Q = q$ and incrementing $m_1(\bs\mu(t-1))$ by $d$ increases the exponent of $Q$ by $\sum_{1 < j} 2d m_j(\bs\mu(t-1)) = 2d (l(\bs\mu(t-1) - m_1(\bs\mu(t-1)))$. Further, $A(t-1, i)$ remains unchanged for $i \geq 2$. So in conclusion, the new value of $B(t-1)$ is larger than the old one by a factor
\[
q^{2d(k - h)} \frac{|U_{h + d}(\mathbb{F}_q)|}{|U_{h}(\mathbb{F}_q)|}.
\]
\end{proof}
\end{toappendix}
\noindent
Now, for the remainder of this section, where we deal with the symplectic and orthogonal groups, suppose that $q$ is the power of an odd prime.

\begin{toappendix}
\noindent
Now, assume $q$ is an odd prime power. In order to describe the sizes of the centralisers of symplectic and orthogonal groups, Wall defines \emph{Hermitian invariants} from elementary divisors of $g$. The Hermitian invariants are bilinear forms whose symmetry groups appear in describing the conjugacy class and centraliser of $g$. Unlike the case of the unitary group, the conjugacy class of $g$ in the general linear group does not uniquely define a conjugacy class in the symplectic or orthogonal group. The Hemitian invariants of $g$ are precisely the additional information that is needed to determine the conjugacy class of $g$. We will not need to use them, so we omit their definition.
\begin{definition}
If $r(t) = \sum_i a_i t^i \in \Phi_{q}$ is a monic irreducible polynomial with coefficients in $\mathbb{F}_{q}$ other than $t$, define
\[
r^*(t) = \frac{t^{\deg(r)}}{r(0)} \sum_i a_i t^{-i},
\]
which is also a monic irreducible polynomial.
\end{definition}
\noindent
A \emph{symplectic signed partition} is a partition $\lambda$ such that
\begin{itemize}
\item for odd $i$, $m_i(\lambda)$ is even,
\item a choice of sign $\epsilon_i \in\{ +,-\}$ for each even $i$ such that $m_i(\lambda) > 0$ is given.
\end{itemize}
These signs record the Hermitian invariants of an element $g \in Sp_{2n}(\mathbb{F}_q)$. Conjugacy classes in $Sp_{2n}(\mathbb{F}_q)$ are indexed by multipartitions $\bs\mu$ of size $2n$ such that $\bs\mu(r) = \bs\mu(r^*)$, except that $\bs\mu(t \pm 1)$ must be symplectic signed partitions. The type of $g$ (i.e. conjugacy class in $GL_{2n}(\mathbb{F}_q)$) is recovered by forgetting the signs of $\bs\mu(t \pm 1)$ to recover ordinary partitions.

\begin{definition}
Let $\bs\mu$ label the conjugacy class of  $g \in Sp_{2n}(\mathbb{F}_q)$. Let $A$ be the following function defined on $\Phi_{q} \times \mathbb{Z}_{>0}$. First of all, if $r \neq t \pm 1$,
\[
A(r, i) =
\left\{
        \begin{array}{ll}
            |U_{m}(\mathbb{F}_{Q^{\frac{1}{2}}})| & \quad \mbox{if } r = r^* \\
            |GL_{m}(\mathbb{F}_{Q})|^{\frac{1}{2}} & \quad \mbox{if } r \neq r^*
        \end{array}
    \right.
\]
where $m = m_i(\bs\mu(r))$ and $Q = q^{\deg(r)}$. For $r = t \pm 1$, we instead have
\[
A(r, i) =
\left\{
        \begin{array}{ll}
            |Sp_{m}(\mathbb{F}_{q})| & \quad \mbox{if $i$ is odd} \\
            q^{\frac{m}{2}}|O_{m}^{\epsilon_i}(\mathbb{F}_{q})| & \quad \mbox{if $i$ is even}
        \end{array}
    \right.
\]
with $m = m_i(\bs\mu(r))$ as before. Note that if $m$ is odd, $O_{m}^{+}(\mathbb{F}_{q})$ and $O_{m}^{-}(\mathbb{F}_{q})$ are isomorphic. Now define
\[
B(r) = Q^{\sum_{i < j} i m_i(\bs\mu(r))m_j(\bs\mu(r)) + \frac{1}{2}\sum_{i \geq 1} (i-1) m_i(\bs\mu(r))^2} \prod_{i \geq 1} A(r, i).
\]
\end{definition}

\begin{proposition}
Then the centraliser of $g \in Sp_{2n}(\mathbb{F}_q)$ has size
\[
\prod_{r \in \Phi_q} B(r).
\]
\end{proposition}
\noindent
As in the unitary case, we will only need the following consequence of this result.
\end{toappendix}

\begin{corollaryrep} \label{symplectic_centraliser_ratio_corollary}
Suppose that $g \in Sp_{2n}(\mathbb{F}_q)$. Let $d$ be a positive even integer. Then $\tilde{g} = \bigl( \begin{smallmatrix}g & 0 \\ 0 & \Id_d \end{smallmatrix}\bigr)$ may be viewed as an element of $Sp_{2n+d}(\mathbb{F}_q)$. The ratio of the sizes of the centralisers of these elements is
\[
\frac{C_{Sp_{2n+d}(\mathbb{F}_q)}(\tilde{g})}{C_{Sp_{2n}(\mathbb{F}_q)}(g)} = q^{d(k-h)}\frac{|Sp_{h + d}(\mathbb{F}_q)|}{|Sp_{h}(\mathbb{F}_q)|},
\]
where $\bs\mu$ is the type of $g$, $k=l(\bs\mu(t-1))$ and $h=m_1(\bs\mu(t-1))$.
\end{corollaryrep}

\begin{toappendix}
\begin{proof}
Under the operation of extending $g$ by a size $d$ identity matrix, we add $(1^{d})$ to the signed symplectic partition $\bs\mu(t-1)$ without changing any of the signs. The rest of the proof is identical to the unitary case.
\end{proof}
\begin{remark} \label{symplectic_char_2_remark}
Corollary \ref{symplectic_centraliser_ratio_corollary} holds even if $q$ is a power of 2, however in this case the description of conjugacy classes of $Sp_{2n}(\mathbb{F}_q)$ is more intricate, and requires applying the more complicated Theorem 3.7.4 of \cite{Wall}. Nevertheless, the same kind of cancellation takes place, and the factor $2^k$ appearing in the formula also cancels because $k$ is defined in terms of the multiplier of $g$, which is determined by the subspace $im(\Id - g)$, which coincides with $im(\Id - \tilde{g})$ upon identifying $\mathbb{F}_q^{2n}$ with a subspace of $\mathbb{F}_q^{2n+d}$.
\end{remark}
\noindent
A \emph{orthogonal signed partition} is a partition $\lambda$ such that 
\begin{itemize}
\item for even $i$, $m_i(\lambda)$ is even,
\item a choice of sign $\epsilon_i \in \{+, -\}$ for each odd $i$ such that $m_i(\lambda)>0$ is given.
\end{itemize}
These signs record the Hermitian invariants of an element $g$ of a finite orthogonal group. Conjugacy classes of orthogonal groups $O_n^{\pm}(\mathbb{F}_q)$ are indexed by multipartitions $\bs\mu$ of size $n$ such that $\bs\mu(r) = \bs\mu(r^*)$, except that $\bs\mu(t \pm 1)$ must be orthogonal signed partitions. Such $\bs\mu$ describes a conjugacy class in exactly one of $O_n^{+}(\mathbb{F}_q)$ or $O_n^{-}(\mathbb{F}_q)$ in a way that can be determined from $\bs\mu$ but we omit here. If $n$ is odd, $O_n^{+}(\mathbb{F}_{q})$ and $O_n^{-}(\mathbb{F}_{q})$ are isomorphic, so we obtain two parametrisations of the conjugacy classes of $O_n(\mathbb{F}_q)$.
\begin{definition}
Let $\bs\mu$ label the conjugacy class of  $g$, either an element of $O_{n}^{+}(\mathbb{F}_q)$ or $O_{n}^{-}(\mathbb{F}_q)$. Let $A$ be the following function defined on $\Phi_{q} \times \mathbb{Z}_{>0}$. First of all, if $r \neq t \pm 1$,
\[
A(r, i) =
\left\{
        \begin{array}{ll}
            |U_{m}(\mathbb{F}_{Q^{\frac{1}{2}}})| & \quad \mbox{if } r = r^* \\
            |GL_{m}(\mathbb{F}_{Q})|^{\frac{1}{2}} & \quad \mbox{if } r \neq r^*
        \end{array}
    \right.
\]
where $m = m_i(\bs\mu(r))$ and $Q = q^{\deg(r)}$. For $r = t \pm 1$, we instead have
\[
A(r, i) =
\left\{
        \begin{array}{ll}
            |O_{m}^{\epsilon_i}(\mathbb{F}_{q})| & \quad \mbox{if $i$ is odd} \\
            q^{-\frac{m}{2}}|Sp_{m}(\mathbb{F}_{Q})| & \quad \mbox{if $i$ is even}
        \end{array}
    \right.
\]
with $m = m_i(\bs\mu(r))$ as before. Now define
\[
B(r) = Q^{\sum_{i < j} i m_i(\bs\mu(r))m_j(\bs\mu(r)) + \frac{1}{2}\sum_{i \geq 1} (i-1) m_i(\bs\mu(r))^2} \prod_{i \geq 1} A(r, i).
\]
\end{definition}
\end{toappendix}
\noindent
This result actually remains true in characteristic two, although the proof is slightly more complicated; see Remark \ref{symplectic_char_2_remark}. Finally, we state the orthogonal version of the previous corollaries.
\begin{corollaryrep}
For $\eta \in W(\mathbb{F}_q)$, let $O_n^\eta(\mathbb{F}_q)$ be the symmetry group of the form with germ $\eta$ (note that the germ determines the parity of $n$). Suppose that $g \in O_n^{\tau}(\mathbb{F}_q)$. Let $d$ be a positive integer, and let $\rho$ be the germ of a nondegenerate symmetric bilinear form on $\mathbb{F}_q^d$. Then $\tilde{g} = \bigl( \begin{smallmatrix}g & 0 \\ 0 & \Id_d \end{smallmatrix}\bigr)$ may be viewed as an element of $O_{n+d}^{\tau \oplus \rho}(\mathbb{F}_q)$. The ratio of the sizes of the centralisers of these elements is
\[
\frac{C_{O_{n+d}^{\tau \oplus \rho}(\mathbb{F}_q)}(\tilde{g})}{C_{O_{n}^{\tau}(\mathbb{F}_q)}(g)} = 
q^{d(k-h)}\frac{|O_{h + d}^{\epsilon_1 \oplus \rho}(\mathbb{F}_q)|}{|O_{h}^{\epsilon_1}(\mathbb{F}_q)|},
\]
where $\bs\mu$ is the type of $g$, $k=l(\bs\mu(t-1))$ and $h=m_1(\bs\mu(t-1))$.\end{corollaryrep}
\begin{proof}
Under the operation of extending $g$ by a size $d$ identity matrix, we add $(1^{d})$ to the signed orthogonal partition $\bs\mu(t-1)$ and adding $\rho$ to the sign $\epsilon_1$ (addition takes place in the Witt ring). (If $\bs\mu(t-1)$ was empty, the new sign corresponds to the germ $\rho$.) The rest of the proof is identical to the unitary case.
\end{proof}
\begin{toappendix}
\noindent
This final case is a little more complicated than the others because the resulting quantity depends on the sign $\epsilon_1$ which is not determined by the germ $\tau$ of the ambient orthogonal group.
\end{toappendix}

\section{General Linear Groups} \label{GL_section}
\noindent
Let $\mathbb{F}_q^{\infty}$ be a countably infinite dimensional vector space over $\mathbb{F}_q$ with basis $\{e_1, e_2, \ldots \}$. We may write
\[
\mathbb{F}_q^\infty = \varinjlim \mathbb{F}_q^n,
\]
where $\mathbb{F}_q^n$ is has basis $\{e_1, e_2, \ldots, e_n\}$, and the maps in the directed limit are the obvious inclusions.
\begin{definition}
Say that a subspace $V$ of $\mathbb{F}_q^\infty$ is \emph{smooth} if it contains $e_i$ for all sufficiently large $i$.
\end{definition}
\noindent
It is immediate that the intersection of two smooth subspaces of $\mathbb{F}_q^\infty$ is again smooth. Note that while smooth subspaces have finite codimension, not every finite codimension subspace is smooth, for example if $L$ is the linear map $L:\mathbb{F}_q^\infty \to \mathbb{F}_q$ defined by $L(e_i) = 1$ for all $i$, then $\ker(L)$ has codimension $1$, but does not contain any $e_i$.
\newline \newline \noindent
Let us view an element of the general linear group $GL_n(\mathbb{F}_q)$ as a matrix with respect to the basis $\{e_1, \ldots, e_n\}$. Then for $n < m$ we have (injective) homomorphisms $\rho_{n,m}: GL_n(\mathbb{F}_q) \to GL_m(\mathbb{F})$ defined by 
\[
\rho_{n,m}(g)(e_i) = \begin{cases} 
          g(e_i) & i \leq n \\
          e_i & i > n
       \end{cases}
\]
which may be viewed extending an $n \times n$ matrix to an $m \times m$ matrix by taking the block sum with the identity matrix, i.e. $\rho_{n,m}(g) = \bigl( \begin{smallmatrix}g & 0 \\ 0 & \Id_{m-n} \end{smallmatrix}\bigr)$.

\begin{definition}
Let $GL_\infty(\mathbb{F}_q)$ be 
\[
\varinjlim GL_n(\mathbb{F}_q),
\]
where the maps in the directed system are the inclusions $\rho_{n,m}$. In other words, $GL_\infty(\mathbb{F}_q)$ consists of invertible matrices (whose entries are indexed by $\mathbb{Z}_{>0} \times \mathbb{Z}_{>0}$) that differ from the identity matrix in finitely many entries. We frequently view $GL_n(\mathbb{F}_q)$ as a subgroup of $GL_\infty(\mathbb{F}_q)$.
\end{definition}
\noindent
It is immediate that if $V$ is a smooth subspace of $\mathbb{F}_q^\infty$ and $g \in GL_\infty(\mathbb{F}_q)$, then $gV$ is also a smooth subspace. Note that since we have chosen a basis to work with, there is a notion of transpose. We write $g^T$ for the transpose of $g \in GL_\infty(\mathbb{F}_q)$.
\begin{definition}
A \emph{bounding triple} is a tuple $(W, g, V)$ where $W$ and $V$ are smooth subspaces of $\mathbb{F}_q^\infty$ and $g \in GL_\infty(\mathbb{F})$ is such that $g$ acts as the identity on $V$ and $g^T$ acts as the identity on $W$. Additionally, we let $BT_n$ be the set of bounding triples $(W,g,V)$ with $\{e_{n+1}, e_{n+2}, \ldots\} \subseteq V, W$.
\end{definition}
\noindent
Of course, $BT_1 \subseteq BT_2 \subseteq \cdots$, and each bounding triple belongs to some $BT_n$, so the $BT_n$ ``filter'' the set of bounding triples. As the next lemma shows, the two spaces $W$, $V$ in a bounding triple $(W,g,V)$ constrain the entries where $g$ can differ from the identity matrix. Thus $W$ and $V$ serve to ``bound'' the behaviour of $g$.
\begin{lemma} \label{Bn_lemma}
Suppose that $(W,g,V) \in BT_n$. Then $g \in GL_n(\mathbb{F}_q)$, i.e. $g$ agrees with the identity matrix starting at the $(n+1)$-th row and $(n+1)$-th column.
\end{lemma}
\begin{proof}
The condition that $\{e_{n+1}, e_{n+2}, \ldots\} \subseteq V$ implies that $g(e_i) = e_i$ for $i > n$, which agrees with the identity matrix. Similarly the condition on $W$ shows $g^T(e_i) = e_i$ for $i > n$, which agrees with the identity matrix. So our $g$ may be viewed as an element of $GL_n(\mathbb{F}_q)$.
\end{proof}

\begin{definition}
Let us say that a bounding triple $(W,g,V)$ is \emph{tight} if $V = \ker(g-1)$ and $W = \ker(g^T - 1)$.
\end{definition}
\noindent
In a bounding triple, $g-1$ and $g^T-1$ act by zero on $V$ and $W$ respectively. So in a tight bounding triple, $V$ and $W$ are as large as possible. It is clear that each $g \in GL_\infty(\mathbb{F}_q)$ is contained in a unique tight bounding triple.

\begin{proposition} \label{triple_conj_prop}
There is an action of $GL_\infty(\mathbb{F}_q)$ on the set of all bounding triples via
\[
x \cdot (W, g, V) = (x^{-T}W, xgx^{-1}, xV),
\]
where $x^{-T} = (x^{-1})^T = (x^T)^{-1}$. We refer to this action as ``conjugation''.
\end{proposition}
\begin{proof}
First we observe that $xgx^{-1}$ acts as the identity on $xV$, and similarly $(xgx^{-1})^T = x^{-T} g^T x^T$ acts as the identity on $x^{-T}W$, so $x \cdot (W,g,V)$ is a bounding triple. Clearly $\Id \cdot (W, g, V) = (W, g, V)$. All that remains to be checked is associativity:
\begin{eqnarray*}
x_1 \cdot ( x_2 \cdot (W, g, V) ) &=& x_1 \cdot (x_2^{-T}W, x_2 g x_2^{-1}, x_2 V) \\
&=& (x_1^{-T}x_2^{-T}W, x_1 x_2 g x_2^{-1} x_1^{-1}, x_1 x_2 V) \\
&=& ((x_1x_2)^{-T}W, x_1x_2 g (x_1x_2)^{-1} , x_1 x_2 V).
\end{eqnarray*}
\end{proof}

\subsection{Conjugacy of Bounding Triples}
\noindent
In order to further understand bounding triples $(W,g,V)$, we introduce a bilinear pairing $\mathbb{F}_q^\infty \times \mathbb{F}_q^\infty \to \mathbb{F}_q$. Here we think of $W$ as being a subspace of the first factor, and $V$ as a subspace of the second factor. In order to be compatible with the conjugacy action, we consider the following action of $GL_\infty(\mathbb{F}_q)$ on $\mathbb{F}_q^\infty \times \mathbb{F}_q^\infty$:
\[
x \cdot (w \otimes v) = (x^{-T}w) \otimes (xv),
\]
so the two factors of $\mathbb{F}_q^\infty$ transform dually to each other.
\begin{definition}
Let $\langle -,- \rangle$ be the bilinear pairing on $\mathbb{F}_q^\infty \times \mathbb{F}_q^\infty$ defined by $\langle e_i, e_j \rangle = \delta_{i,j}$.
\end{definition}
\noindent
By construction, $\langle -,- \rangle$ is $GL_\infty(\mathbb{F}_q)$-invariant. In particular, for smooth subspaces $W, V \subseteq \mathbb{F}_q^\infty$, we have a (possibly degenerate) pairing $W \times V \to \mathbb{F}_q$ by restricting $\langle -,- \rangle$. Although the action of $x \in GL_\infty(\mathbb{F}_q)$ may not preserve $W$ or $V$, we have $\dim(W \cap V^\perp) = \dim((x^{-T}W) \cap (xV)^\perp)$ and $\dim(W^\perp \cap V) = \dim((x^{-T}W)^\perp \cap (xV))$.

\begin{lemma} \label{first_decomposition_lemma}
Suppose that $(W,g,V) \in BT_n$. There is a vector space decomposition $\mathbb{F}_q^\infty = U_1 \oplus U_2 \oplus U_3$ such that:
\begin{itemize}
\item $U_1, U_2 \subseteq \mathbb{F}_q^n$ and $U_3$ contains $e_{n+1}, e_{n+2}, \ldots$,
\item $U_2 \oplus U_3 = V$,
\item $U_2 = V \cap W^\perp$ and the restricted pairing $W \times U_3 \to \mathbb{F}_q$ is non-degenerate on the right.
\end{itemize}
\end{lemma}

\begin{proof}
Since $W$ contains $e_{n+1}, e_{n+2}, \ldots$, we have $W^\perp \subseteq \mathbb{F}_q^n$. So if we take $U_2 = V \cap W^\perp$, we get $U_2 \subseteq \mathbb{F}_q^n$. Then we take $U_3$ to be any complement to $U_2$ in $V$ that contains $e_{n+1}, e_{n+2}, \ldots$. The pairing $W \times U_3$ is non-degenerate on the right because $U_3 \cap W^\perp \subseteq V \cap W^\perp = U_2$ and $U_2$ intersects $U_3$ trivially. Finally, we pick $U_1$ to be any complement to $V$ in $\mathbb{F}_q^\infty$ that is contained inside $\mathbb{F}_q^n$ (which exists because $V$ contains $e_{n+1}, e_{n+2}, \ldots$).
\end{proof}

\begin{proposition} \label{standard_shape_prop}
Suppose that $(W,g,V) \in BT_n$. Let $a = \dim(W \cap V^\perp)$, $b = \dim(V \cap W^\perp)$, and let $c$ be the rank of the pairing $\langle -,- \rangle$ restricted to $(W \cap \mathbb{F}_q^n) \times (V \cap \mathbb{F}_q^n)$. Then there is $x \in GL_n(\mathbb{F}_q)$ and a decomposition $\mathbb{F}_q^\infty = E_1 \oplus E_2 \oplus E_3 \oplus E_4$ such that $x \cdot (W, g, V) = (W^\prime , g^\prime, V^\prime)$, where:
\begin{itemize}
\item $E_1 = \mathbb{F}_q\{e_1, \ldots, e_a\}$,
\item $E_2 = \mathbb{F}_q\{e_{a+1}, \ldots, e_{n-b-c}\}$,
\item $E_3 = \mathbb{F}_q\{e_{n-b-c+1}, \ldots, e_{n-c}\}$,
\item $E_4 = \mathbb{F}_q\{e_{n-c+1}, e_{n-c+2}, \ldots\}$,
\item $V^\prime = E_3 \oplus E_4$,
\item $W^\prime = E_1 \oplus E_4$.
\end{itemize}
Moreover, when viewed as an element of $GL_n(\mathbb{F}_q)$, $g^\prime$ has the block matrix form
\[
\begin{bmatrix}
\Id & 0 & 0 & 0\\
C & A & 0 & 0 \\
D & B & \Id & 0 \\
0 & 0 & 0 & \Id
\end{bmatrix}
\]
where the sizes of the blocks are $a, n-a-b-c, b, c$, and $A,B,C,D$ are appropriately sized matrices.
\end{proposition}

\begin{proof}
We first perform a preliminary conjugation to pin down $V^\prime$ before we conjugate a second time to obtain the required form of $W^\prime$. Consider the decomposition $\mathbb{F}_q^\infty = U_1 \oplus U_2 \oplus U_3$ provided by Lemma \ref{first_decomposition_lemma}, noting that $V = U_2 \oplus U_3$. We may construct a matrix $x^{\prime \prime}$ that describes a change of basis between the standard basis of $\mathbb{F}_q^\infty$ and a basis obtained by concatenating bases of $U_1, U_2, U_3$. In particular, we may take the basis of $U_3$ to consist of $c$ vectors in $\mathbb{F}_q^n$ followed by $e_{n+1}, e_{n+2}, \ldots$. Together with the fact that $U_1, U_2 \subseteq \mathbb{F}_q^n$ this guarantees that $x^{\prime \prime} \in GL_n(\mathbb{F}_q)$. Now we consider $(W^{\prime \prime}, g^{\prime \prime}, V^{\prime \prime}) = x^{\prime \prime} \cdot (W, g, V)$.
\newline \newline \noindent
By construction, $V^{\prime \prime} = \mathbb{F}_q\{e_{r+1}, e_{r+2} \ldots\}$, where $r = \dim(U_1) = \codim(V) = n - b - c$. So in particular, $V^{\prime \prime} = E_3 \oplus E_4$. This implies that $(V^{\prime \prime})^\perp = \mathbb{F}_q\{e_1, \ldots, e_r\} = E_1 \oplus E_2$. With respect to the decomposition $\mathbb{F}_q^n = \mathbb{F}_q^r \oplus (V^{\prime \prime} \cap \mathbb{F}_q^n)$, $g^{\prime \prime}$ has the block form 
\[
\begin{bmatrix}
M_1 & 0 \\
M_2 & \Id
\end{bmatrix},
\]
because $g^{\prime \prime}$ fixes $V^{\prime \prime}$ pointwise. Now we will conjugate by a second element $x^\prime$ to get $x^\prime \cdot (W^{\prime \prime}, g^{\prime \prime}, V^{\prime \prime}) = (W^\prime, g^\prime, V^\prime)$. We take $x^\prime \in GL_n(\mathbb{F}_q)$ to have block form
\[
\begin{bmatrix}
P_1 & 0 \\
P_2 & P_3
\end{bmatrix},
\]
so that $V^{\prime} = x^\prime V^{\prime \prime} = V^{\prime \prime}$. This allows $(x^{\prime})^{-T}$ to be an arbitrary invertible block-upper-triangular matrix. We use this freedom to control $W^\prime = (x^\prime)^{-T}W^{\prime \prime}$. The fact that $(x^\prime)^{-T}$ is block upper-triangular means we must preserve $\mathbb{F}_q^r = (V^{\prime \prime})^\perp$, but there are no other restrictions on the change of basis. So we choose to move $W^{\prime \prime} \cap \mathbb{F}_q^r = W^{\prime \prime} \cap (V^{\prime \prime})^\perp$ to $\mathbb{F}_q \{e_1, \ldots, e_a\} = E_1$ (which remains contained in $\mathbb{F}_q^r$). We move the rest of $W \cap \mathbb{F}_q^n$ to $\mathbb{F}_q \{e_{n-c+1}, \ldots, e_n\} = E_4 \cap \mathbb{F}_q^n$.
\newline \newline \noindent
Putting this all together, if $x = x^\prime x^{\prime \prime}$, then $x \cdot (W, g, V) = (W^\prime, g^\prime, V^\prime)$ where we have $V^\prime = E_3 \oplus E_4$ and $W^\prime = E_1 \oplus E_4$. Now, with respect to the decomposition $\mathbb{F}_q^n = E_1 \oplus E_2 \oplus E_3 \oplus E_4$, $g^\prime$ has the block matrix form
\[
\begin{bmatrix}
\Id & 0 & 0 & 0\\
C & A & 0 & 0 \\
D & B & \Id & 0 \\
0 & 0 & 0 & \Id
\end{bmatrix}
\]
because $g^\prime$ acts by the identity on $V^\prime = E_3 \oplus E_4$, and because $(g^\prime)^{T}$ acts by the identity on $W^\prime = E_1 \oplus E_4$ (since $(g^\prime)^{-T}$ does).
\end{proof}

\begin{definition}
We say that a bounding triple $(W^\prime,g^\prime,V^\prime) \in BT_n$ in the form described in Proposition \ref{standard_shape_prop} has \emph{standard shape}.
\end{definition}

\begin{proposition} \label{gl_inf_gl_n_conjugacy_prop}
Two bounding triples in $BT_n$ that are conjugate by an element of $GL_\infty(\mathbb{F}_q)$ are conjugate by an element of $GL_n(\mathbb{F}_q)$.
\end{proposition}
\begin{proof}
Suppose we are given bounding triples $(W_1, g_1, V_1), (W_2, g_2, V_2) \in BT_n$ that are conjugate by an element of $GL_\infty(\mathbb{F}_q)$. By Proposition \ref{standard_shape_prop}, they are $GL_n(\mathbb{F}_q)$-conjugate to triples in standard shape. So to prove our bounding triples are $GL_n(\mathbb{F}_q)$-conjugate, we may assume that $(W_1, g_1, V_1)$ and $(W_2, g_2, V_2)$ are in standard shape.
\newline \newline \noindent
Since $(W_1, g_1, V_1)$ and $(W_2, g_2, V_2)$ are $GL_\infty(\mathbb{F}_q)$-conjugate, they are conjugate by an element of $GL_{m}(\mathbb{F}_q)$ for some $m$. We may take $m \geq n$, which guarantees $(W_1, g_1, V_1), (W_2, g_2, V_2) \in BT_m$. Note that if $(W,g,V) \in BT_n$ has block sizes $a,n-a-b-c,b,c$ in its standard shape, then viewing it as an element of $BT_m$ instead gives a standard shape with block sizes $a, n-a-b-c, b, m-n+c$. This is because $a = \dim(W \cap V^\perp)$ and $b = \dim(V \cap W^\perp)$ are unchanged, while the rank of the bilinear form on $(W \cap \mathbb{F}_q^m) \otimes (V \cap \mathbb{F}_q^m)$ is $m-n$ larger than the rank of the bilinear form on $(W \cap \mathbb{F}_q^n) \otimes (V \cap \mathbb{F}_q^n)$ since both $V$ and $W$ contain $e_{n+1}, \ldots, e_m$.
\newline \newline \noindent
We deduce from this that $(W_1, g_1, V_1)$ and $(W_2, g_2, V_2)$ have blocks of the same size when put in standard shape (as elements of $BT_m$). In particular, this implies that $W_1 = W_2$ and $V_1 = V_2$ since these spaces are spanned by certain basis vectors determined by the block sizes. We recall the block structure of our matrices:
\[
g_1 = \begin{bmatrix}
\Id & 0 & 0 & 0\\
C_1 & A_1 & 0 & 0 \\
D_1 & B_1 & \Id & 0 \\
0 & 0 & 0 & \Id
\end{bmatrix}, \hspace{10mm}
g_2 = \begin{bmatrix}
\Id & 0 & 0 & 0\\
C_2 & A_2 & 0 & 0 \\
D_2 & B_2 & \Id & 0 \\
0 & 0 & 0 & \Id
\end{bmatrix},
\]
where the block sizes are $a, n-a-b-c, b, m-n+c$. Now we consider $x \in GL_{m}(\mathbb{F}_q)$ which conjugates one bounding triple to the other. In terms of the same block structure, in order for $x^{-T}$ to preserve $W=W_1=W_2$, $x$ must have the form
\[
\begin{bmatrix}
* & 0 & 0 & *\\
* & * & * & * \\
* & * & * & * \\
* & 0 & 0 & *
\end{bmatrix},
\]
where $*$ indicates an unconstrained entry. Similarly, for $x$ to preserve $V=V_1=V_2$, it must take the form
\[
\begin{bmatrix}
* & * & 0 & 0\\
* & * & 0 & 0 \\
* & * & * & * \\
* & * & * & *
\end{bmatrix}.
\]
Combining these conditions, we may take
\[
x = 
\begin{bmatrix}
x_{11} & 0 & 0 & 0 \\
x_{21} & x_{22} & 0 & 0 \\
x_{31} & x_{32} & x_{33} & x_{34} \\
x_{41} & 0 & 0 & x_{44}
\end{bmatrix},
\]
and we note that such $x$ has determinant $\det(x) = \det(x_{11}) \det(x_{22}) \det(x_{33}) \det(x_{44}) \neq 0$. The condition $xg_1x^{-1} = g_2$ is equivalent to $x (g_1  - \Id) = (g_2 - \Id) x$, which we now explicitly work out:
\begin{eqnarray*}
x(g_1-\Id) &=& \begin{bmatrix}
0 & 0 & 0 & 0\\
x_{22} C_1 & x_{22}(A_1 - \Id) & 0 & 0 \\
x_{32} C_1 + x_{33} D_1 & x_{32}(A_1 - \Id) + x_{33} B_1 & 0 & 0 \\
0 & 0 & 0 & 0
\end{bmatrix}, \\
(g_2 - \Id) x &=& \begin{bmatrix}
0 & 0 & 0 & 0\\
C_2 x_{11} + (A_2 - \Id) x_{21} & (A_2 - \Id) x_{22} & 0  & 0 \\
D_2 x_{11} + B_2 x_{21} & B_2 x_{22} & 0 & 0 \\
0 & 0 & 0 & 0
\end{bmatrix}.
\end{eqnarray*}
The equality of these two matrices imposes no restrictions on the entries $x_{41}$, $x_{34}$, $x_{44}$ (or $x_{31}$), so we may replace $x_{41}$ and $x_{34}$ with zero matrices, and set $x_{44}$ to be the identity matrix. This modified matrix $x_{new}$ has block form 
\[
x_{new} = \begin{bmatrix}
x_{11} & 0 & 0 & 0 \\
x_{21} & x_{22} & 0 & 0 \\
x_{31} & x_{32} & x_{33} & 0 \\
0 & 0 & 0 & \Id
\end{bmatrix},
\]
and is invertible because $\det(x_{11})\det(x_{22})\det(x_{33}) \neq 0$. Therefore $x_{new}$ may be viewed as an element of $GL_{n-c}(\mathbb{F}_q) \subseteq GL_n(\mathbb{F}_q)$. Additionally, $x_{new}g_1 = g_2 x_{new}$ and $x_{new}V = V$, $x_{new}^{-T}W = W$, as needed.
\end{proof}

\subsection{Products of Bounding Triples} \label{bounding_triples_subsection}

\begin{definition}
We define a multiplication on bounding triples via the formula
\[
(W_1, g_1, V_1) \times (W_2, g_2, V_2) = (W_1 \cap W_2, g_1 g_2, V_1 \cap V_2).
\]
\end{definition}
\noindent
Since $g_1$ and $g_2$ both act on $V_1 \cap V_2$ as the identity (and similarly for $g_1^T$ and $g_2^T$ on $W_1 \cap W_2$), the product of two bounding triples is again a bounding triple.

\begin{proposition}
Multiplication of bounding triples is associative and also $GL_\infty(\mathbb{F}_q)$-equivariant.
\end{proposition}
\begin{proof}
Associativity follows from associativity of intersections and of multiplication in $GL_\infty(\mathbb{F}_q)$. Let $x \in GL_\infty(\mathbb{F}_q)$. Then
\begin{eqnarray*}
(x \cdot (W_1, g_1, V_1)) \times (x \cdot (W_2, g_2, V_2)) &=& (x^{-T}W_1, xg_1x^{-1}, xV_1) \times (x^{-T}W_2, xg_2x^{-1}, xV_2) \\
&=& (x^{-T}W_1 \cap x^{-T}W_2, xg_1g_2x^{-1}, xV_1 \cap x V_2) \\
&=& (x^{-T}(W_1 \cap W_2), x g_1 g_2 x^{-1}, x (V_1 \cap V_2)) \\
&=& x \cdot ((W_1, g_1, V_1) \times  (W_2, g_2, V_2)).
\end{eqnarray*}
\noindent Here we have used the fact that $xV_1 \cap x V_2= x (V_1 \cap V_2)$ (and similarly for $W_1$ and $W_2$), which follows from the fact that $x$ is a bijection.
\end{proof}
\noindent
Hence, the set of all bounding triples forms a semigroup, and each $BT_n$ is also a semigroup.

\begin{proposition} \label{finite_product_prop}
There are only finitely many ways in which a given bounding triple can be expressed as a product of two bounding triples.
\end{proposition}

\begin{proof}
Suppose $(W_3, g_3, V_3)$ is given, and we would like to find solutions to $(W_1, g_1, V_1) \times (W_2, g_2, V_2) = (W_3, g_3, V_3)$, in other words,
\[
(W_1 \cap W_2, g_1 g_2, V_1 \cap V_2) = (W_3, g_3, V_3).
\]
Let $n$ be such that $W_3$ and $V_3$ both contain $e_{n+1}, e_{n+2}, \ldots$. Then $W_1, W_2, V_1, V_2$ must also contain these elements. There are only finitely many subspaces containing $e_{n+1}, e_{n+2}, \ldots$ (such subspaces are in bijection with subspaces of $\mathbb{F}_q^{\infty} / \mathbb{F}_q \{e_{n+1}, e_{n+2}, \ldots\}$ which is finite dimensional). By Lemma \ref{Bn_lemma}, it also follows that $g_1$ and $g_2$ correspond to elements of the finite group $GL_n(\mathbb{F}_q)$.
\end{proof}

\begin{definition}
Let $\mathcal{A}$ be the set of functions from the set of all bounding triples to $\mathbb{Z}$. It is an abelian group with pointwise addition. We equip $\mathcal{A}$ with the following convolution product. If $f_1, f_2 \in \mathcal{A}$, then
\[
(f_1 * f_2) (W_3, g_3, V_3) = \sum_{(W_1, g_1, V_1) \times (W_2, g_2, V_2) = (W_3, g_3, V_3)} f_1(W_1, g_1, V_1) f_2(W_2, g_2, V_2).
\]
\end{definition}
\noindent
By Proposition \ref{finite_product_prop}, the sum is finite, and therefore well defined.

\begin{example}
If $f_1$ and $f_2$ are the indicator functions of $(W_1, g_1, V_1)$ and $(W_2, g_2, V_2)$ respectively, then $f_1 * f_2$ is the indicator function of $(W_1, g_1, V_1) \times (W_2, g_2, V_2)$. As a result, we may view $\mathcal{A}$ as a completed version of the monoid algebra (over $\mathbb{Z}$) of bounding triples.
\end{example}
\noindent
There is an action of $GL_\infty(\mathbb{F}_q)$ on $\mathcal{A}$ via
\[
(x \cdot f) (W, g, V) = f( x^{-1} \cdot (W, g, V)).
\]
Because the multiplication is equivariant for the action of $GL_\infty(\mathbb{F}_q)$, the product of two $GL_\infty(\mathbb{F}_q)$-invariant elements is again invariant.
\begin{definition}
Let $\mathcal{A}^{GL_\infty}$ be the subspace of $\mathcal{A}$ consisting of elements that are invariant for the action of $GL_\infty(\mathbb{F}_q)$ and are supported on finitely many $GL_\infty(\mathbb{F}_q)$-orbits of bounding triples. We call $\mathcal{A}^{GL_\infty}$ the \emph{general linear Ivanov-Kerov algebra}.
\end{definition}

\begin{proposition} \label{Ivanov_Kerov_subalg_proposition}
We have that $\mathcal{A}^{GL_\infty}$ is a subalgebra of $\mathcal{A}$.
\end{proposition}

\begin{proof}
All we need to check is that the product of the indicator functions of two $GL_\infty(\mathbb{F}_q)$-orbits of bounding triples is supported on finitely many $GL_\infty(\mathbb{F}_q)$-orbits. To see this, we consider the product of two elements in the orbits of fixed bounding triples:
\begin{eqnarray*}
(W_3, g_3, V_3) &=& (x_1 \cdot (W_1, g_1, V_1)) \times (x_2 \cdot (W_2, g_2, V_2)) \\
&=& (x_1^{-T}W_1 \cap x_2^{-T}W_2, x_1g_1x_1^{-1}x_2g_2x_2^{-1}, x_1V_1 \cap x_2 V_2).
\end{eqnarray*}
This bounding triple is contained in $BT_n$ for some sufficiently large $n$. By Proposition \ref{standard_shape_prop}, we may conjugate it to lie in $BT_{n-c}$, where $c$ is the rank of the paring
$(W_3 \cap \mathbb{F}_q^n) \times (V_3 \cap \mathbb{F}_q^n) \to \mathbb{F}_q$. But the rank of this pairing is bounded below by $n - \codim(W_3) - \codim(V_3) \geq n - \codim(W_1) - \codim(W_2) - \codim(V_1) - \codim(V_2)$. We conclude that any bounding triple $(W_3, g_3, V_3)$ arising in this way is conjugate to an element of $BT_m$, where $m = \codim(W_1) + \codim(W_2) + \codim(V_1) + \codim(V_2)$ is independent of $x_1$ and $x_2$. Since $BT_m$ contains finitely many elements, it intersects finitely many orbits.
\end{proof}

\section{Specialisation Homomorphisms} \label{specitalisation_section}

\begin{proposition}
For any $n \in \mathbb{Z}_{\geq 0}$, there is a surjective homomorphism $\Psi_n: \mathcal{A} \to \mathbb{Z}GL_n(\mathbb{F}_q)$ defined by
\[
\Psi_n(f) = \sum_{(W, g, V) \in BT_n} f(W,g,V)g.
\]
\end{proposition}
\begin{proof}
As argued in Proposition \ref{finite_product_prop}, there are only finitely many $(W, g, V)$ appearing in the sum, so the formula is well defined. It is clear that the formula respects addition. For multiplication we find
\begin{eqnarray*}
\Psi_n(f_1 * f_2) &=& \sum_{(W, g, V) \in BT_n} \sum_{(W_1, g_1, V_1) \times (W_2, g_2, V_2) = (W, g, V)} f_1(W_1, g_1, V_1) f_2(W_2, g_2, V_2) g \\
&=& \left(\sum_{(W_1, g_1, V_1) \in BT_n} f_1(W_1, g_1, V_1) g_1 \right) \left( \sum_{(W_2, g_2, V_2) \in BT_n} f_2(W_2, g_2, V_2) g_2 \right) \\
&=& \Psi_n(f_1) \Psi_n(f_2).
\end{eqnarray*}
Here we used the fact that $(W_1, g_1, V_1) \times (W_2, g_2, V_2) \in BT_n$ if and only if $(W_1, g_1, V_1) \in BT_n$ and $(W_2, g_2, V_2) \in BT_n$. Now consider the indicator function $f$ of $(W,g,V)$ where $g \in GL_n(\mathbb{F}_q)$ is arbitrary while $V = W = \mathbb{F}_q\{ e_{n+1}, e_{n+2}, \ldots \}$. By definition $\Psi_n(f) = g$, and surjectivity follows by linearity.
\end{proof}

\begin{proposition}
The image of $\mathcal{A}^{GL_\infty}$ under $\Psi_n$ is precisely the centre $Z(\mathbb{Z}GL_n(\mathbb{F}_q))$.
\end{proposition}

\begin{proof}
Firstly we check that the map $\Psi_n$ is $GL_n(\mathbb{F}_q)$-equivariant, where the action on $\mathbb{Z}GL_n(\mathbb{F}_q)$ is conjugation. Let $x \in GL_n(\mathbb{F}_q)$. Then
\begin{eqnarray*}
\Psi_n( x \cdot f) &=& \sum_{(W, g, V) \in BT_n} f(x^{T}W, x^{-1}gx, x^{-1}V) g \\
&=& \sum_{(W, g, V) \in BT_n} f(W, g, V) xgx^{-1} \\
&=& x \Psi_n(f) x^{-1}.
\end{eqnarray*}
In order to reindex the sum we used the fact that $x^{-T}W$ and $xV$ contain $\{e_{n+1}, e_{n+2}, \ldots\}$ if and only if $W$ and $V$ do, and also that $g \in GL_n(\mathbb{F}_q)$ if and only if $xgx^{-1} \in GL_n(\mathbb{F}_q)$.
\newline \newline \noindent
Since a $GL_\infty(\mathbb{F}_q)$-orbit is a union of $GL_n(\mathbb{F}_q)$ orbits, a function $f \in \mathcal{A}^{GL_\infty}$ is constant on $GL_n(\mathbb{F}_q)$-orbits. It follows that $\Psi_n(f)$ is fixed under conjugation by $GL_n(\mathbb{F}_q)$, so it is in the centre of the group algebra. To see that this map surjects onto the centre, it is enough to show that conjugacy-class sums are in the image. For $x \in GL_n(\mathbb{F}_q)$, we take $f$ to be the indicator function of tight bounding triples $(W, g, V)$ where $g$ is conjugate to $x$ in $GL_\infty(\mathbb{F}_q)$ (recall that for tight bounding triples, $g$ determines $W$ and $V$, so there is only one for each $g$). Proposition \ref{gl_inf_gl_n_conjugacy_prop} shows that the only such bounding triples $(W, g, V)$ in $BT_n$ are those where $g$ and $x$ are conjugate by elements of $GL_n(\mathbb{F}_q)$. Then $\Psi_n(f)$ is precisely the sum of all $g$ that are conjugate to $x$ and contained in $GL_n(\mathbb{F}_q)$.
\end{proof}

\begin{proposition} \label{specialisation_function_gl_proposition}
Let $f$ be the indicator function of the $GL_\infty(\mathbb{F}_q)$-orbit of $(W, g, V)$, where $(W, g, V) \in BT_n$. Let $a,b,c$ have the same meanings as in Proposition \ref{standard_shape_prop}, which shows that up to conjugation we may assume $(W,g,V) \in BT_{n-c}$. Let $\bs\mu$ be the type of $g$ viewed as an element of $GL_{n-c}(\mathbb{F}_q)$. Suppose that $m \geq n-c$, and let $Cl(g)$ be the sum of all elements in $GL_{m}(\mathbb{F}_q)$ that are conjugate to $g$. Then
\[
\Psi_m(f) = K \qbinom{m-n+c+h}{h}_q q^{(m-n+c)(2k-h-a-b)} Cl(g),
\]
where  $k = l(\bs\mu(t-1))$ and $h = m_1(\bs\mu(t-1))$, and $K$ is an integer independent of $m$.
\end{proposition}
\begin{proof}
It is immediate that that $\Psi_m(f) = P \cdot Cl(g)$ for some integer $P$. The coefficient $P$ is equal to the number of different pairs of subspaces $(W^\prime, V^\prime)$ such that $(W^\prime, g, V^\prime)$ and $(W, g, V)$ are $GL_{m}(\mathbb{F}_q)$-conjugate. Explicitly, for some $x \in GL_{m}(\mathbb{F}_q)$,
\begin{eqnarray*}
V^\prime &=& xV \\
g &=& xgx^{-1} \\
W^\prime &=& x^T W.
\end{eqnarray*}
The condition $g = xgx^{-1}$ is saying that $x$ is in the centraliser $C_{GL_{m}(\mathbb{F}_q)}(g)$. So $P$ is the size of the orbit of the pair $(W,V)$ under the action of $C_{GL_{m}(\mathbb{F}_q)}(g)$. Then, the orbit-stabiliser relation implies
\[
P = \frac{|C_{GL_{m}(\mathbb{F}_q)}(g)|}{|C_{GL_{m}(\mathbb{F}_q)}(g) \cap \mathrm{Stab}_m(W,V)|},
\]
where $\mathrm{Stab}_m(W,V)$ is the subgroup of $GL_{m}(\mathbb{F}_q)$ stabilising the pair $(W,V)$ under the action $x \cdot (W,V) = (x^{-T}W, xV)$. By the same computation as in Proposition \ref{gl_inf_gl_n_conjugacy_prop}, taking $g$ to be in standard shape 
\[
g=
\begin{bmatrix}
\Id & 0 & 0 & 0\\
C & A & 0 & 0 \\
D & B & \Id & 0 \\
0 & 0 & 0 & \Id
\end{bmatrix},
\]
we find $\mathrm{Stab}_m(V,W)$ consists of block matrices of the form
\[
x = 
\begin{bmatrix}
x_{11} & 0 & 0 & 0\\
x_{21} & x_{22} & 0 & 0 \\
x_{31} & x_{32} & x_{33} & x_{34} \\
x_{41} & 0 & 0 & x_{44}
\end{bmatrix},
\]
where the block sizes are $a, n-a-b-c, b, m-n+c$. To compute the size of the intersection of $C_{GL_{m}(\mathbb{F}_q)}(g)$ and $\mathrm{Stab}_m(W,V)$, we use the fact that the condition of commuting with $g$ does not depend on the blocks $x_{14}$, $x_{43}$, $x_{44}$. This means that any element $x \in C_{GL_{m}(\mathbb{F}_q)}(g) \cap \mathrm{Stab}_m(V,W)$ is obtained by extending some element
\[
x_{small} = 
\begin{bmatrix}
x_{11} & 0 & 0\\
x_{21} & x_{22} & 0\\
x_{31} & x_{32} & x_{33}
\end{bmatrix} \in GL_{n-c}(\mathbb{F}_q) \cap C_{GL_{m}(\mathbb{F}_q)}(g) \cap \mathrm{Stab}_m(V,W)
\]
with arbitrary choices of $x_{41} \in \mathbb{F}_q^{a (m-n+c)}$, $x_{34} \in \mathbb{F}_q^{(m-n+c)b}$, $x_{44} \in GL_{m-n+c}(\mathbb{F}_q)$. So if we let
\[
N = |GL_{n-c}(\mathbb{F}_q) \cap C_{GL_{m}(\mathbb{F}_q)}(g) \cap \mathrm{Stab}_m(V,W)|
\]
be the number of possible choices of $x_{small}$, then
\[
|C_{GL_{m}(\mathbb{F}_q)}(g) \cap \mathrm{Stab}_m(V,W)| = N q^{(a+b) (m-n+c)} \prod_{i=1}^{m-n+c}(q^{m-n+c}-q^{m-n+c-i}),
\]
where $N$ does not depend on $m$.
\newline \newline \noindent
The conjugacy class of $g$ in $GL_{m}(\mathbb{F}_q)$ is labelled by the multipartition $\bs \mu \cup (1^{m-n+c})_{t-1}$. By Corollary \ref{gln_centraliser_corollary}, 
\[
\frac{|C_{GL_{m}(\mathbb{F}_q)}(g)|}{|C_{GL_{n-c}(\mathbb{F}_q)}(g)|} = q^{(m-n+c)(2k + m-n+c)} \prod_{i=h+1}^{h+m-n+c} (1 - q^{-i}),
\]
Since $N$ is the order of a subgroup of $C_{GL_{n-c}(\mathbb{F}_q)}(g)$, it follows that $K = |C_{GL_{n-c}(\mathbb{F}_q)}(g)|/N$ is an integer independent of $m$. Finally, we have
\begin{eqnarray*}
P &=& \frac{|C_{GL_{m}(\mathbb{F}_q)}(g)|}{|C_{GL_{m}(\mathbb{F}_q)}(g) \cap \mathrm{Stab}_m(W,V)|} \\
&=&
\frac{|C_{GL_{n-c}(\mathbb{F}_q)}(g)|}{N q^{(a+b) (m-n+c)} \prod_{i=1}^{m-n+c}(q^{m-n+c}-q^{m-n+c-i})}\frac{|C_{GL_{m}(\mathbb{F}_q)}(g)|}{|C_{GL_{n-c}(\mathbb{F}_q)}(g)|} \\
&=&
K \frac{q^{(m-n+c)(2k + m-n+c)} \prod_{i=h+1}^{h+m-n+c} (1 - q^{-i})}{q^{(a+b) (m-n+c)} \prod_{i=1}^{m-n+c}(q^{m-n+c}-q^{m-n+c-i})} \\
&=&
K q^{(m-n+c)(2k-a-b-h)}
\frac{\prod_{i=h+1}^{m-n+c+h} (q^i - 1)}{\prod_{i=1}^{m-n+c} (q^i - 1)} \\
&=& 
K  \qbinom{m-n+c+h}{h}_q q^{(m-n+c)(2k-a-b-h)}.
\end{eqnarray*}
\end{proof}

\begin{lemma} \label{lower_case_specialisation_correctness_lemma}
The formula in Proposition \ref{specialisation_function_gl_proposition} is correct not just for $m \geq n-c$, but for all integers $m \geq 0$ (we interpret $Cl(g)$ to be zero if there are no elements of $GL_{m}(\mathbb{F}_q)$ conjugate to $g$).
\end{lemma}

\begin{proof}
If $m < n-c$, then $\Psi_m(f) = 0$ for the following reason. Assume that $(W, g, V) \in BT_{n-c}$ so that the rank of the pairing on $(W \cap \mathbb{F}_q^{n-c}) \otimes (V \cap \mathbb{F}_q^{n-c})$ is zero. Then if $(W,g,V) \in BT_m$ with $m < n-c$, let the rank of the pairing on $(W \cap \mathbb{F}_q^m) \otimes (V \cap \mathbb{F}_q^m)$ be $c^\prime \geq 0$. Thus the rank of the pairing on $(W \cap \mathbb{F}_q^{n-c}) \otimes (V \cap \mathbb{F}_q^{n-c})$ is $c^\prime - (m-n+c) > 0$, a contradiction. On the other hand, the formula in Proposition \ref{specialisation_function_gl_proposition} involves the terms $\qbinom{m-n+c+h}{h}_q$ and $Cl(g)$, and we now argue that one of these must vanish if $m < n-c$.
\newline \newline \noindent
For $Cl(g)$ to be nonzero in $\mathbb{Z}GL_{m}(\mathbb{F}_q)$, we must have that $|\bs\mu| - m_1(\bs\mu(t-1)) \leq m$. But $|\bs\mu| = n-c$ and $m_1(\bs\mu(t-1)) = h$, so we get $m-n+c+h \geq 0$. Since $m < n-c$, we in fact have $ 0 \leq m-n+c+h < h$, from which it follows that $\qbinom{m-n+c+h}{h}_q = 0$.
\end{proof}

\begin{proposition} \label{exponent_nonnegativity_proposition}
The quantity $2k-a-b-h$ appearing in Proposition \ref{specialisation_function_gl_proposition} is non-negative.
\end{proposition}
\begin{proof}
Viewing $g$ as an element of $GL_{n-c}(\mathbb{F}_q)$, we have the block matrix form
\[
g = \begin{bmatrix}
\Id & 0 & 0\\
C & A & 0 \\
D & B & \Id 
\end{bmatrix}
\]
with blocks of size $a$, $n-a-b-c$, $b$. We need to determine the conjugacy classes in $GL_{n-c}(\mathbb{F}_q)$ that have an element of the above form. Problems of this nature are addressed in Chapter 4, Section 3 of \cite{Macdonald1995}, where it is shown that this reduces to multiplication of Hall-Littlewood polynomials whose indexing partitions come from the types of the diagonal blocks (in this case $\Id, A, \Id$). In more detail, let $P$ be the parabolic subgroup consisting the the block shape of the matrix above. Then inflating the class function which is the indicator function of the conjugacy class of $\mathrm{diag}(\Id, A, \Id) \in GL_{a}(\mathbb{F}_q) \times GL_{n-a-b-c}(\mathbb{F}_q) \times GL_{b}(\mathbb{F}_q)$ to $P$, and the inducing up to $GL_{n-c}(\mathbb{F}_q)$ we obtain a class function supported on elements of $GL_{n-c}(\mathbb{F}_q)$ conjugate to a matrix of the above form. It turns out that this parabolic induction operation is described by Hall polynomials $g_{\mu, \nu}^\lambda$. If $\pi_1, \pi_2$ are indicator functions of conjugacy classes of types $\bs\mu^{(1)}, \bs\mu^{(2)}$ respectively, then the parabolic induction of $\pi_1$ and $\pi_2$ contains the indicator function of $\bs\mu$ with coefficient
\[
\prod_{r \in \Phi_q} g_{\bs\mu^{(1)}(r), \bs\mu^{(2)}(r)}^{\bs\mu(r)} (q^{\deg(r)}),
\]
by Equation 3.3 of Chapter 4, Section 3 of \cite{Macdonald1995}. In our situation, we have a three-fold product, where two of the factors (corresponding to identity matrices) have types vanishing away from $r(t) = t-1$, where they take the values $(1^a)$ and $(1^b)$. So the above product reduces just to the case where $r(t)=t-1$, and by Equation 4.6 of Chapter 2, Section 4 of \cite{Macdonald1995}, $g_{\mu, (1^d)}^\lambda$ is zero unless $\lambda/\mu$ is a vertical strip of size $d$.
\newline \newline \noindent
If $\bs\nu$ is the type of the matrix $A$ and $\bs\mu$ be the type of $g$, then $\bs \mu(r) = \bs \nu(r)$ for all irreducible polynomials $r$ different from $t-1$. In terms of Young diagrams, $\bs \mu(t-1)$ is obtained from $\bs \nu (t-1)$ by adding $a$ boxes, no two in the same row, and then adding $b$ boxes, no two in the same row. For each row (i.e. part) of $\bs\nu(t-1)$, let us consider the cases where either 2, 1, or 0 boxes are added in order to obtain $\bs\mu(t-1)$.
\newline \newline \noindent
To each row we associate a number computed as follows. We begin with 2, and subtract the number of boxes that were added to the row in the above procedure. If the result has length 1, we also subtract 1. Then the sum of all these numbers will be exactly $2k - a- b- h$. The only way one of these numbers could be negative is if two boxes were added and the resulting length was 1. This is impossible, so each individual number is non-negative. In particular, their sum, $2k - a - b- h$ is non-negative.
\end{proof}

\begin{remark}
In fact,  $2k - a - b - h$ is zero if the original bounding triple was tight because this forces two boxes to be added to each part of $\bs\nu(t-1)$, so the number associated to each row is zero.
\end{remark}

\section{Stable Centres for General Linear Groups}\label{stable_centres_for_general_linear_groups}
\noindent
We now use the theory we have developed to explicitly address stability properties of centres of group algebras of $GL_n(\mathbb{F}_q)$ as $n$ varies. First of all, we need a way to discuss conjugacy classes of these groups for all $n$ at once.

\begin{definition}[Subsection 2.3 \cite{Wan_Wang}]\label{modified_type_definition}
If $g \in GL_n(\mathbb{F}_q)$, we define the \emph{modified type} of $g$ to be the multipartition $\bs \nu$ obtained from the type, $\bs \mu$, of $g$ in the following way. For each irreducible polynomial $r$ other than $t-1$, $\bs\nu(r) = \bs\mu(r)$. On the other hand, the partition $\bs\nu(t-1)$ is obtained form $\bs\mu(t-1)$ by subtracting $1$ from each part (i.e. we delete the first column in the Young diagram representation of $\bs\mu(t-1)$). 
\end{definition}

\begin{example} \label{modified_type_example}
The identity element of $GL_n(\mathbb{F}_q)$ has type $\bs\mu$ where $\bs\mu(t-1)=(1^n)$ and $\bs\mu(r)$ is the empty partition for $r \neq t-1$, and so its modified type is the empty multipartition.
\end{example}
\begin{example} \label{modified_type_example_2}
Consider the matrix $g$ of type $\bs\mu$ in Example \ref{basic_example}. The modified type $\bs\nu$ of $g$ is $\bs\nu(t-1) = (2)$ and $\bs\nu(f) = \bs\mu(f)$ for all other polynomials $f(t) \neq t$ irreducible over $\mathbb{Q}$.
\end{example}
\noindent
The reason for defining the modified type is that it is invariant under the embedding $GL_n(\mathbb{F}_q) \rightarrow GL_{n+1}(\mathbb{F}_q)$. We may recover the type $\bs\mu$ from the modified type $\bs\nu$ provided $n$ is known: we must add $1$ to each part of $\bs\nu(t-1)$ (possibly including some parts of size zero), so that the total size increases to $n$. This means that $n - |\bs\nu|$ parts of $\bs\nu(t-1)$ are incremented, and so this is possible precisely when the number of parts of $\bs\nu(t-1)$ is less than or equal to $n - |\bs\nu|$. (If this condition does not hold, then $GL_n(\mathbb{F}_q)$ does not contain any elements of modified type $\bs \nu$.)

\begin{definition}
Let $X_{\bs\nu, n}$ denote the sum of all elements of modified type $\bs\nu$ in $GL_n(\mathbb{F}_q)$, viewed as an element of $Z(\mathbb{Z}GL_n(\mathbb{F}_q))$.
\end{definition}
\noindent
Note that $X_{\bs\nu, n}$ is either the sum of elements in a conjugacy class, or zero, according to whether $GL_n(\mathbb{F}_q)$ does, or does not, have elements of modified type $\bs\nu$. In either case, it is a central element of the group algebra.

\begin{theorem} \label{gl_structure_constant_theorem}
There is a family of elements $r_{\bs\mu, \bs\nu}^{\bs\lambda} \in \mathcal{R}_q$ interpolating the structure constants of $Z(\mathbb{Z}GL_{m}(\mathbb{F}_q))$ in the following way:
\[
X_{\bs\mu,m} X_{\bs\nu, m} = \sum_{\bs\lambda} r_{\bs\mu, \bs\nu}^{\bs\lambda}([m]_q) X_{\bs\lambda, m}.
\]
Here, $\bs\mu, \bs\nu$ are arbitrary multipartitions and the sum ranges over all multipartitions.
\end{theorem}
\begin{proof}
Let us work in $\mathcal{A}^{GL_\infty}$. Let $f_{\bs\mu}$ be the indicator function of the $GL_\infty(\mathbb{F}_q)$-orbit of a tight bounding triple for an element of modified type $\bs\mu$, and similarly define $f_{\bs\nu}$. Using the results of Section \ref{specitalisation_section} we argue as follows. Firstly, $\Psi_m(f_{\bs\mu}) = X_{\bs\mu, m}$ and $\Psi_m(f_{\bs\nu}) = X_{\bs\nu, m}$, so 
\[
\Psi_m(f_{\bs\mu}f_{\bs\nu}) = X_{\bs\mu, m}X_{\bs\nu, m}.
\]
Now, $f_{\bs\mu}f_{\bs\nu}$ a linear combination (with integer coefficients) of indicator functions of $GL_\infty(\mathbb{F}_q)$-orbits of some bounding triples $(W,g,V)$. But by Proposition \ref{specialisation_function_gl_proposition} such an indicator function specialises to 
\[
K \qbinom{m-n+c+h}{h}_q q^{(m-n+c)(2k-h-a-b)} X_{\bs\lambda, m},
\]
where $\bs\lambda$ is the modified type of $g$, $K$ is an integer independent of $m$, and $a,b,c,h,k,n$ are the usual quantities associated to the bounding triple $(W,g,V)$. Now,
\[
q^{(m-n+c)(2k-a-b-h)} = (1+(q-1)[m]_q)^{2k-a-b-h} q^{-(n-c)(2k-a-b-h)}
\]
is a polynomial in $[m]_q$ with coefficients in $\mathbb{Z}[q,q^{-1}]$ (and so it is the evaluation of an element of $\mathcal{R}_q$ at $x=[m]_q$). By Lemma \ref{shifted_qivp_lemma}, $\qbinom{m-n+c+h}{h}_q$ is the evaluation of some element of $\mathcal{R}_q$ at $x=[m]_q$. Combining these, we conclude that $K \qbinom{m-n+c+h}{h}_q q^{(m-n+c)(2k-h-a-b)}$ is obtained by evaluating an element of $\mathcal{R}_q$ at $[m]_q$. Finally, summing over all orbits where $g$ has modified type $\bs\lambda$, we obtain the required element $r_{\bs\mu, \bs\nu}^{\bs\lambda}$.
\end{proof}
\noindent
Some comments are in order. First of all, following the discussion in Subsection \ref{qivp_subsection}, we may instead view $r_{\bs\mu, \bs\nu}^{\bs\lambda}$ as a polynomial in $q^m$ (rather than $[m]_q$) if we wish. Secondly, suppose that we wished to consider $Z(\mathbb{F}_q GL_n(\mathbb{F}_q))$, for example, to study the representation theory of $GL_n(\mathbb{F}_q)$ in defining characteristic. We would want to evaluate the quantities 
\[
K \qbinom{m-n+c+h}{h}_q q^{(m-n+c)(2k-h-a-b)} X_{\bs\lambda, m}\]
modulo $q$. Note that $2k-a-b-h \geq 0$ (Proposition \ref{exponent_nonnegativity_proposition}) and $m-n+c \geq 0$ whenever the above quantity is nonzero (Lemma \ref{lower_case_specialisation_correctness_lemma}). This means that only non-negative powers of $q$ arise, so there is no issue in passing to coefficients in $\mathbb{F}_q$. However, the polynomials $r_{\bs\mu, \bs\nu}^{\bs\lambda}$ will have coefficients in $\mathbb{Z}[q,q^{-1}]$, and hence might not be defined over $\mathbb{F}_q$. This is a facet of how the representation theory of the general linear group in defining characteristic is very different from non-defining characteristic.

\begin{definition}
Let $\mathrm{FH}_q^{GL}$ be the free $\mathcal{R}_q$-module with basis given by symbols $K_{\bs\mu}$ for multipartitions $\bs\mu$. We equip $\mathrm{FH}_q^{GL}$ with a bilinear multiplication defined on basis elements via
\[
K_{\bs\mu} K_{\bs\nu} = \sum_{\bs\lambda} r_{\bs\mu, \bs\nu}^{\bs\lambda} K_{\bs\lambda},
\]
where $r_{\bs\mu, \bs\nu}^{\bs\lambda} \in \mathcal{R}_q$ are the elements from Theorem \ref{gl_structure_constant_theorem}. We call $\mathrm{FH}_q^{GL}$ the \emph{general linear Farahat-Higman algebra}.
\end{definition}

\begin{corollary}\label{gl_special_hom_cor}
There is a ``specialisation'' homomorphism $ \Theta_n: \mathrm{FH}_q^{GL} \to Z(\mathbb{Z}GL_n(\mathbb{F}_q))$ defined by $\Theta_n(K_{\bs\mu}) = X_{\bs\mu,n}$ and by evaluating the coefficients (elements of $\mathcal{R}_q$) at $[n]_q$.
\end{corollary}

\begin{proof}
This is the content of Theorem \ref{gl_structure_constant_theorem}.
\end{proof}

\begin{proposition}\label{glfh_is_nice_prop}
We have that $\mathrm{FH}_q^{GL}$ is an associative, commutative, unital $\mathcal{R}_q$-algebra.
\end{proposition}
\begin{proof}
To check these properties, we appeal to the maps $\Theta_n$. Since $Z(\mathbb{Z}GL_n(\mathbb{F}_q))$ is commutative, we have $X_{\bs\mu, n}X_{\bs\nu, n} = X_{\bs\nu, n}X_{\bs\mu, n}$. This means that
\[
\sum_{\bs\lambda} r_{\bs\mu, \bs\nu}^{\bs\lambda}([n]_q) X_{\bs\lambda, n} = \sum_{\bs\lambda} r_{\bs\nu, \bs\mu}^{\bs\lambda}([n]_q) X_{\bs\lambda, n}.
\]
For $n$ large enough that $X_{\bs\lambda, n}$ is nonzero, we get $r_{\bs\mu, \bs\nu}^{\bs\lambda}([n]_q) = r_{\bs\nu, \bs\mu}^{\bs\lambda}([n]_q)$. But the set of $[n]_q$ for $n$ sufficiently large is a Zariski-dense subset of $\mathbb{Z}$, so in fact we have an equality of polynomials $r_{\bs\mu, \bs\nu}^{\bs\lambda} = r_{\bs\nu, \bs\mu}^{\bs\lambda}$, which proves commutativity. Similarly considering \[
\Theta_n((K_{\bs\mu}K_{\bs\nu})K_{\bs\lambda}) = \Theta_n(K_{\bs\mu}(K_{\bs\nu}K_{\bs\lambda}))
\]
which follows form associativity in $Z(\mathbb{Z}GL_n(\mathbb{F}_q))$, we deduce associativity in $\mathrm{FH}_q^{GL}$. Similarly again, we see that $X_{\varnothing}$ is the identity element, because $\Theta_n(X_{\varnothing})$ is the identity element of $Z(\mathbb{Z}GL_n(\mathbb{F}_q))$ (see Example \ref{modified_type_example}).
\end{proof}

\begin{lemma} \label{degree_bound_lemma}
The degree of the polynomial $r_{\bs\mu, \bs\nu}^{\bs\lambda}$ is at most $2(|\bs\mu| + |\bs\nu| - |\bs\lambda|)$.
\end{lemma}

\begin{proof}
In Theorem \ref{gl_structure_constant_theorem}, $r_{\bs\mu, \bs\nu}^{\bs\lambda}$ is calculated by considering the product of tight bounding triples corresponding to modified types $\bs\mu$ and $\bs\nu$. We seek to understand what happens when the result has modified type $\bs\lambda$. Let $(W_1, g_1, V_1)$ be a tight bounding triple where the modified type of $g_1$ is $\bs\mu$ and let $(W_2, g_2, V_2)$ be a tight bounding triple where the modified type of $g_2$ is $\bs\nu$. Note that we have $\codim(W_1) = \codim(V_1) = |\bs\mu|$ and $\codim(W_2) = \codim(V_2) = |\bs\nu|$. So, for
\[
(W_1, g_1, V_1) \times (W_2, g_2, V_2) = (W_1 \cap W_2, g_1g_2, V_1 \cap V_2)
\]
we have $\codim(W_1 \cap W_2) \leq |\bs\mu| + |\bs\nu|$ and $\codim(V_1 \cap V_2) \leq |\bs\mu| + |\bs\nu|$. Now let us view this product of bounding triples in standard shape:
\[
g_1g_2 =
\begin{bmatrix}
\Id & 0 & 0\\
C & A & 0 \\
D & B & \Id
\end{bmatrix}
\]
where as usual, the block sizes are $a, n-a-b-c, b$. The codimension of $W = W_1 \cap W_2$ is $n-a-c$ and the codimension of $V = V_1 \cap V_2$ is $n-b-c$. Let $\bs\rho$ be the (unmodified) type of this $(n-c) \times (n-c)$ matrix. Hence $|\bs\rho| = n-c$ and if $\bs\lambda$ is modified type of the above matrix, we have $|\bs\lambda| = |\bs\rho| - l(\bs\rho(t-1))$.
\newline \newline \noindent
The contribution of the indicator function of the orbit of $(W_1 \cap W_2, g_1g_2, V_1 \cap V_2)$ to $r_{\bs\mu, \bs\nu}^{\bs\lambda}([m]_q)$ is
\[
K \qbinom{m-n+c+h}{h}_q q^{(m-n+c)(2k-h-a-b)}
\]
which (viewed as a polynomial in $[m]_q$, or equivalently, $q^m$) has degree $2k-a-b$. Recall that $k = l(\bs\rho(t-1))$. We must show that $2k-a-b \leq 2 (|\bs\mu| + |\bs\nu| - |\bs\lambda|)$. Right away we point out that
\[
2(|\bs\mu| + |\bs\nu|) \geq \codim(W_1 \cap W_2) + \codim(V_1 \cap V_2) = 2n - a - b - 2c = 2(n-c-k) + 2k-a-b.
\]
Upon identifying $n-c-k = |\bs\rho| - l(\bs\rho(t-1)) = |\bs\lambda|$, we obtain the required inequality.
\end{proof}

\begin{proposition}[Theorem 3.4, \cite{Wan_Wang}]\label{gl_assoc_graded}
The algebra $\mathrm{FH}_q^{GL}$ is filtered, where $K_{\bs\mu}$ is in filtration degree $|\bs\mu|$. Moreover the structure constants of the associated graded algebra are integers (rather than arbitrary elements of $\mathcal{R}_q$).
\end{proposition}

\begin{proof}
This follows from Lemma \ref{degree_bound_lemma}, noting that in the associated graded algebra, the structure constants are $r_{\bs\mu,\bs\nu}^{\bs\lambda}$ where $|\bs\mu| + |\bs\nu| = |\bs\lambda|$ and hence $r_{\bs\mu,\bs\nu}^{\bs\lambda}$ is a degree zero polynomial.
\end{proof}

\begin{remark}
There is a canonical algebra map $\sigma: \mathcal{A}^{GL_\infty} \to \mathrm{FH}_q^{GL}$ such that
\[
\Theta_n \circ \sigma = \Psi_n.
\]
It sends the indicator function of the orbit of $(W,g,V)$ to $p(x) K_{\bs\mu}$ where $\bs\mu$ is the modified type of $g$ and $p(x) \in \mathcal{R}_q$ obeys
\[
p([m]_q) = K \qbinom{m-n+c+h}{h}_q q^{(m-n+c)(2k-h-a-b)}
\]
as in Proposition \ref{specialisation_function_gl_proposition}. However, it is not injective as we now show.
\newline \newline \noindent
Consider $(U_1,g,U_2)$ where $g$ is the identity, and $U_1 = \mathbb{F}_q^\infty$, while $U_2 = \mathbb{F}_q\{e_2, e_3, \ldots\}$. If $f_1$ is the indicator function of the orbit of this triple, then 
\[
\Psi_n(f_1) = |Gr(n-1,n)| \mathrm{Id},
\]
where $Gr(n-1,n)$ is the Grassmannian of $(n-1)$-dimensional subspaces inside $\mathbb{F}_q^n$; each such space corresponds to a choice of $U_2^\prime$ such that $(U_1,g,U_2^\prime) \in B_n$ and $(U_1,g,U_2^\prime)$ is conjugate to $(U_1,g,U_2)$. (Incidentally, $|Gr(n-1,n)| =  [n]_q$, so $p(x)=x$.) However, we could also let $f_2$ be the indicator function of the orbit of $(U_2, g, U_1)$. The same argument shows
\[
\Psi_n(f_2) = |Gr(n-1,n)| \mathrm{Id},
\]
and it follows that $\sigma(f_1) = \sigma(f_2)$. But $f_1 \neq f_2$ in $\mathcal{A}^{GL_\infty}$. So unlike the symmetric group case, the Ivanov-Kerov algebra $\mathcal{A}^{GL_\infty}$ and the general linear Farahat-Higman algebra $\mathrm{FH}_q^{GL}$ are not interchangeable.
\end{remark}

\section{Classical Groups} \label{classical_groups_section}
\noindent
We now make the appropriate modifications to the constructions and arguments from Section \ref{GL_section} to obtain analogous results for classical groups. First of all, we make particular choices of sesquilinear forms.
\begin{definition} \label{standard_form_definition}
Let $n$ be a positive integer, and let $x_i$ and $y_i$ denote the components of vectors $x$ and $y$. The \emph{standard Hermitian form} on $V_n = \mathbb{F}_{q^2}^{n}$ is defined by 
\[
B_n(x,y) = \sum_{i=1}^n x_i \sigma(y_i),
\]
where $\sigma(z) = z^q$. The \emph{standard alternating form} on $V_n = \mathbb{F}_q^{2n}$, is defined by 
\[
B_n(x,y) = \sum_{i=1}^{n} x_{2i-1} y_{2i} - x_{2i}y_{2i-1}.
\]
The \emph{standard positive-symmetric form} on $V_n = \mathbb{F}_q^{2n}$ is defined by
\[
B_n(x,y) = \sum_{i=1}^n x_{2i-1} y_{2i} + x_{2i}y_{2i-1},
\]
The \emph{standard negative-symmetric form} on $V_n = \mathbb{F}_q^{2n}$ is defined by
\[
B_n(x,y) = x_1 y_1 - m x_2y_2 + \sum_{i=2}^n x_{2i-1} y_{2i} + x_{2i}y_{2i-1},
\]
where $m$ is a fixed non-square in $\mathbb{F}_q$. (This is chosen to have non-zero germ.)
\newline \noindent
The \emph{standard odd-symmetric form} on $V_n = \mathbb{F}_q^{2n+1}$ is defined by
\[
B_n(x,y) = x_1 y_1 + \sum_{i=1}^n x_{2i} y_{2i+1} + x_{2i+1}y_{2i},
\]
When we wish to refer to all five cases at once, we will say that $B_n$ is a \emph{standard sesquilinear form} and let $V_n$ be the corresponding vector space as defined above.
\end{definition}
\noindent
The names positive, negative, and odd are not standard, but we need some way to refer to these cases separately in order to study the groups $O_{2n}^{+}(\mathbb{F}_q)$, $O_{2n}^-(\mathbb{F}_q)$ and $O_{2n+1}(\mathbb{F}_q)$. We make these particular choices of symmetric forms for technical reasons; what is really essential is that in each case passing from $n$ to $n+1$ amounts to taking the direct sum of $V_n$ with a hyperbolic plane to get $V_{n+1}$. As a result we obtain inclusions of certain classical groups. Let $B_n$ be a standard sesquilinear form. Then we see that the inclusion of $V_n$ into $V_{n+1}$ preserves the form, i.e. if $v,w \in V_n$, then $B_{n}(v,w) = B_{n+1}(v,w)$. In each case, $V_{n+1} = V_n \oplus V_{n}^\perp$ and we get that $G_{B_n}(V_n) \subseteq G_{B_{n+1}}(V_{n+1})$ is the subgroup fixing $V_n^\perp$ pointwise. (If $g \in G_{B_{n+1}}(V_{n+1})$ satisfies $g V_n^\perp = V_n^\perp$, then also $g V_n = V_n$ because $g$ preserves the sesquilinear form and $V_n^{\perp \perp} = V_n$.)
\begin{definition}
We define
\[
V_\infty = \varinjlim V_n,
\]
which depends on which kind of form we are considering (because $V_n$ does). There is a sesquilinear form $B_\infty$ defined on $V_\infty$ as follows. If $v,w \in V_n$, then
\[
B_\infty(v,w) = B_m(v,w)
\]
for any $m \geq n$ (this is independent of $m$ by the discussion above). We also say that a subspace of $V_\infty$ is \emph{smooth} if it contains $V_n^\perp$ for some $n$.
\end{definition}
\noindent
Implicitly we have equipped $V_\infty$ with a basis coming from the standard bases of the $V_n$. As in the general linear case, a smooth subspace has finite codimension in $V_\infty$, but not every subspace with finite codimension is smooth.
\begin{lemma}
We have $V_\infty = V_n \oplus V_n^\perp$ and inside $V_\infty$, $V_n^{\perp \perp} = V_n$ for any $n$.
\end{lemma}
\begin{proof}
This is straightforward to check in each case; $V_n$ has a basis consisting of the first $n$ (or $2n$, or $2n+1$ depending on the case) basis vectors of $V_\infty$, while the remaining basis vectors form a basis $V_n^\perp$.
\end{proof}
\begin{definition}
We define $U_\infty(\mathbb{F}_q)$, $Sp_\infty(\mathbb{F}_q)$, $O_\infty^+(\mathbb{F}_q)$, $O_\infty^-(\mathbb{F}_q)$, $O_\infty(\mathbb{F}_q)$ to be $\varinjlim G_{B_n}(V_n)$ for each of the five respective cases in Definition \ref{standard_form_definition}. The maps are the inclusions arising from the inclusions $V_n \hookrightarrow V_{n+1}$. In order to treat the five cases uniformly, we use the notation $G_\infty(\mathbb{F}_q)$ to refer to any of the symmetry groups of the sesquilinear forms.
\end{definition}
\noindent
In each case we may view these infinite groups as consisting of matrices that differ from the identity matrix in finitely many entries and preserve a suitable sesquilinear form on $V_\infty$.

\begin{definition}
Let us fix one of the five different choices of $G_\infty(\mathbb{F}_q)$. A \emph{sesquilinear bounding pair} is a pair $(g, V)$, where $g \in G_\infty(\mathbb{F}_q)$, and $V$ is a smooth subspace of $V_\infty$ that is fixed pointwise by $g$.
\end{definition}
\noindent
\begin{definition}
Let us say that a bounding pair $(g,V)$ is \emph{tight} if $V = \ker(g-1)$.
\end{definition}
\noindent
As in the general linear case, in a tight bounding pair, $V$ is as large as possible. So each $g$ is contained in a unique tight bounding pair. The following proposition is directly analogous to Proposition \ref{triple_conj_prop}, so we omit the proof.
\begin{proposition}
There is an action of $G_\infty(\mathbb{F}_q)$ on the set of bounding pairs, defined by
\[
x \cdot (g, V) = (xgx^{-1}, xV).
\]
\end{proposition}

\begin{definition}
Let $BP_n$ be the set of bounding pairs $(g, V)$ such that $V \supseteq V_n^\perp$. 
\end{definition}
\noindent
Now we will see why the fact that $g$ preserves a sesquilinear allows us to work with pairs rather than triples.
\begin{lemma}
If $(g, V) \in BP_n$, then we may view $g$ as an element of $G_{B_n}(V_n)$, i.e. $g$ fixes $V_n^\perp$ pointwise and $g V_n \subseteq V_n$.
\end{lemma}
\begin{proof}
By the definition of $BP_n$, it follows that $g$ fixes $V_n^\perp \subseteq V$ pointwise. Now, suppose that $v \in V_n$. We have $B(v,w) = 0$ for all $w$ in $V_n^\perp$. So also $B(gv, gw) = B(gv, w) = 0$ because $g$ preserves the sesquilinear form and $gw = w$. We conclude that $gv \in V_n^{\perp\perp} = V_n$. So $g$ preserves the decomposition $V_\infty = V_n \oplus V_n^\perp$ and acts by the identity on the second summand.
\end{proof}

\begin{definition}
We say $(g,V) \in BP_n$ has \emph{standard shape} with respect to a decomposition $V_n = W_1 \oplus W_2 \oplus W_3 \oplus W_4$ if the following conditions hold:
\begin{itemize}
\item $W_3 = V \cap V^\perp$ is the radical of $B_n$ when restricted to $V$,
\item $V \cap V_n = W_3 \oplus W_4$, where $B_n$ restricts to a non-degenerate form on $W_4$,
\item $V^\perp = W_2 \oplus W_3$, where $B_n$ is non-degenerate when restricted to $W_2$,
\item $W_1$ is orthogonal to $W_1 \oplus W_2 \oplus W_4$ and $B_n$ induces a non-degenerate pairing between $W_1$ and $W_3$.
\end{itemize}
\end{definition}

\begin{lemma}
For any bounding pair $(g, V) \in BP_n$, we may find a decomposition $V_n = W_1 \oplus W_2 \oplus W_3 \oplus W_4$ putting $(g, V)$ in standard shape.
\end{lemma}

\begin{proof}
Note that $W_3 = V \cap V^\perp$ is automatically determined; since $V_n^\perp \subseteq V$, we have $V^\perp \subseteq V_n$ which means $W_3 \subseteq V_n$. We may then pick $W_4$ to be a complement of $W_3$ in $V \cap V_n$ (similarly to Proposition \ref{sesquilinear_structure_prop} this forces $B_n$ to be non-degenerate on $W_4$). Also we may take $W_2$ to be a complement to $W_3$ in $V^\perp$. Since the orthogonal space to $V^\perp$ in $V_n$ is $V^{\perp \perp} \cap V_n = V \cap V_n$, the radical of $B_n$ when restricted to $V^\perp$ is $V^\perp \cap V \cap V_n = W_3 \cap V_n = W_3$. This implies that $B_n$ is non-degenerate when restricted to $W_2$. Now we concern ourselves with $W_1$. First of all, $B_n$ is non-degenerate on $W_2 \oplus W_4$ since it is non-degenerate on each space, and the two spaces are orthogonal. This means that $(W_2 \oplus W_4)^\perp \oplus (W_2 \oplus W_4) = V_n$. 
\newline \newline \noindent
Let $d_1 = \dim(V \cap V_n)$ and $d_2 = \dim(V \cap V^\perp) = \dim(W_3)$. Then $\dim(W_4) = d_1 - d_2$, also $\dim(V^\perp) = \dim(V_n)-d_1$ and hence $\dim(W_2) = \dim(V_n)-d_1-d_2$. We conclude
\[
\dim((W_2 \oplus W_4)^\perp \cap V_n) = \dim(V_n) - \dim(W_2) - \dim(W_4) = 2 d_2 = 2 \dim(W_3).
\]
Note that $W_3$ is contained inside $(W_2 \oplus W_4)^\perp$, so we may apply Lemma \ref{lagrangian_subspace_lemma} to find a complement $W_1$ to $W_3$ in $(W_2 \oplus W_4)^\perp$ such that $B_n$ restricts to zero on $W_1$.
\end{proof}

\begin{lemma} \label{random_sesquilinear_realisation_lemma}
Suppose some $(g, V) \in BP_n$ has standard shape with respect to the decomposition $V_n = W_1 \oplus W_2 \oplus W_3 \oplus W_4$. Then for any sesquilinear form $B^\prime$ on $W_1$, there exists a linear map $J: W_1 \to W_3$ such that
\[
B^\prime (u,v) = B_n(u, Jv) + B_n(Ju, v),
\]
for any $u,v \in W_1$.
\end{lemma}
\begin{proof}
Fix a basis $u_i$ of $W_1$ and take $W_3$ to have the dual basis $u_i^\prime$ under $B_n$, so that $B_n(u_i, u_j^\prime) = \delta_{ij}$ and $B(u_j^\prime, u_i) = \varepsilon \delta_{ij}$, where $\varepsilon = 1$ if $B_n$ is Hermitian or symmetric, and $\varepsilon = -1$ if $B_n$ is alternating. An element of $W_1$ may be viewed as column vector whose $i$-th entry is the coefficient of $u_i$. Let us assume that $J = D J^\prime$ where $D: W_1 \to W_3$ maps $u_i$ to $u_i^\prime$ and $J^\prime : W_1 \to W_1$ is a matrix to be determined. Then for $u,v \in W_1$ the sesquilinear form
\[
B_n(u, Jv) + B_n(Ju, v) = u^T \sigma(J^\prime) \sigma(v) + \varepsilon u^T (J^\prime)^T \sigma(v) = u^T(\sigma(J^\prime) + \varepsilon (J^\prime)^T)\sigma(v)
\]
has matrix $\sigma(J^\prime) + \varepsilon (J^\prime)^T$. Finally, we show that the matrix of an arbitrary sesquilinear form $B^\prime$ on $W_1$ may be expressed in this way.
\newline \newline \noindent
The matrix $M$ of a Hermitian form obeys $\sigma(M) = M^T$. We take $J^\prime$ to be the upper-triangular matrix whose entries above the diagonal agree with those of $M^T$. The diagonal entries $M_{ii}$ obey $M_{ii} = \sigma(M_{ii})$, and hence by Lemma \ref{trace_value_lemma} we may choose the diagonal entry $J_{ii}^\prime$ to be some $x \in \mathbb{F}_{q^2}$ such that $x + \sigma(x) = M_{ii}$.
\newline \newline \noindent
The matrix $M$ of an alternating form obeys $M^T = -M$, and the diagonal entries are zero (in characteristic 2 this follows from $B_n(v,v) = 0$). So we may simply take $J^\prime$ to be the upper-triangular matrix whose diagonal entries are zero, and above-diagonal entries agree with $M$.
\newline \newline \noindent
The matrix $M$ of symmetric form obeys $M^T = M$. When working with symmetric bilinear forms, we have assumed that the characteristic is odd. So in this case we may take $J^\prime = \frac{1}{2}M$.
\end{proof}

\begin{remark} \label{simplification_remark}
Suppose that $V_n = W_1 \oplus W_2 \oplus W_3 \oplus W_4$ puts some element $g$ in standard shape. If we have two vectors $u, v \in V_n$ with
$u = u_1 + u_2 + u_3 + u_4$ and $v = v_1 + v_2 + v_3 + v_4$, where $u_i, v_i \in W_i$ for $i=1,2,3,4$. Then we have
\[
B_n(u,v) = B_n(u_1, v_3) + B_n(u_2, v_2) + B_n(u_3, v_1) + B_n(u_4, v_4)
\]
by taking into account which $W_i$ and $W_j$ are orthogonal.
\end{remark}

\begin{lemma} \label{sesquilinear_shape_lemma}
Suppose that $(g, V) \in BP_n$ is in standard shape with respect to $W_1 \oplus W_2 \oplus W_3 \oplus W_4$. We may write $g$ as a block matrix with respect to this decomposition, in which case it takes the form
\[
g = \begin{pmatrix}
\Id & 0 & 0 & 0\\
g_{21} & g_{22} & 0 & 0\\
g_{31} & g_{32} & \Id & 0 \\
0 & 0 & 0 & \Id
\end{pmatrix}.
\]
\end{lemma}
\begin{proof}
Since $g$ acts by the identity on $V \cap V_n = W_3 \oplus W_4$, the last two columns must take the stated form. Additionally, since $W_2$ is orthogonal to $V$, $gW_2$ is orthogonal to $gV = V$. It follows that $gW_2 \subset V^\perp = W_2 \oplus W_3$, which explains the structure of the second column. Suppose that $u \in W_1$ and $v \in W_3$. Then
\[
B_n(u,v) = B_n(gu, gv) = B_n(gu, v).
\]
Thus $B_n(gu-u, v) = 0$ for all $v \in W_3$ and hence $gu-u \in W_3^\perp = W_2 \oplus W_3 \oplus W_4$. In particular the component of $gu$ in $W_1$ equals the component of $u$ in $W_1$, so the top-left entry of the matrix is the identity. Finally, the bottom-left entry must be zero because $g$ fixes $W_4$, so $g$ must preserve $W_4^\perp = W_1 \oplus W_2 \oplus W_3$.
\end{proof}

\begin{lemma} \label{sesquilinear_extension_lemma}
Suppose that $(g, V) \in BP_n$ is in standard shape with respect to $V_n = W_1 \oplus W_2 \oplus W_3 \oplus W_4$. Any $x \in G_{B_n}(V_n)$ such that $xV = V$ may be written in the block matrix form
\[
x = \begin{pmatrix}
x_{11} & 0 & 0 & 0\\
x_{21} & x_{22} & 0 & 0\\
x_{31} & x_{32} & x_{33} & x_{34} \\
x_{41} & 0 & 0 & x_{44}
\end{pmatrix}.
\]
Now suppose that $x$ is any matrix of the above form, not necessarily in $G_{B_n}(V_n)$. For any values of the entries $x_{11} \in \hom(W_1, W_1)$ and $x_{41} \in \hom(W_1, W_4)$, there exists $J \in \hom(W_1, W_3)$ such that $x$ is an element of $G_{B_n}(V_n)$ if and only if the three following conditions hold:
\begin{itemize}
\item Firstly, the following matrix is an element of $G_{B_n}(V_n)$:
\[
x_{new} = \begin{pmatrix}
x_{11} & 0 & 0 & 0\\
x_{21} & x_{22} & 0 & 0\\
x_{31} + J & x_{32} & x_{33} & 0 \\
0 & 0 & 0 & \Id
\end{pmatrix}.
\]
\item Secondly, $x_{44} \in G_{B_n}(W_4)$. 
\item Thirdly, $x_{34} \in \hom(W_4, W_3)$ has the value uniquely determined by the equation
\[
B_n(x_{11}u_1, x_{34}v_4) + B_n(x_{41}u_1, x_{44}v_4) = 0
\]
for $u_1 \in W_1$ and $v_4 \in W_4$.
\end{itemize}
\end{lemma}

\begin{proof}
The condition $xV = V$ becomes $x(W_3 \oplus W_4) = W_3 \oplus W_4$. Since $x$ is an isometry, we also have $xV^\perp = V^\perp$ which becomes $x(W_2 \oplus W_3) = W_2 \oplus W_3$. Similarly to the proof of Proposition \ref{gl_inf_gl_n_conjugacy_prop}, this implies that $x$ has the stated form.
\newline \newline \noindent
Let us consider compare the actions of $x$ and $x_{new}$ on the spaces $W_i$ to determine when they preserve $B_n$. Both matrices act equally on $W_2 \oplus W_3$, so one preserves $B_n$ on this space if and only if the other does. We have that $x_{new}$ trivially preserves the form on $W_4$, where it acts as the identity, while if $u_4, v_4 \in W_4$, then by Remark \ref{simplification_remark},
\[
B_n(x u_4, x v_4) = B_n(x_{44} u_4, x_{44} v_4).
\]
This shows that $x$ preserves $B_n$ on $W_4$ if and only if $x_{44} \in G_{B_n}(W_4)$. Note that $(W_2 \oplus W_3)^\perp = W_3 \oplus W_4$, so the shapes of $x$ and $x_{new}$ force them both to preserve orthogonality between $W_2 \oplus W_3$ and $W_4$. 
\newline \newline \noindent
Finally, let us consider $W_1$. Suppose that $u_i \in W_i$ ($i=1,2,3,4$). Applying Remark \ref{simplification_remark} to $B_n(xu, xv)$ and $B_n(x_{new}u, x_{new}v)$ for suitable $u$ and $v$:
\begin{eqnarray*}
B_n(x u_1, x u_2) &=& B_n(x_{11} u_1, x_{32} u_2) + B_x(x_{21} u_1, x_{22}u_2) = B_n(x_{new} u_1, x_{new} u_2), \\
B_n(x u_1, x u_3) &=& B_n(x_{11} u_1, x_{33} u_3) = B_n(x_{new} u_1, x_{new} u_3),
\end{eqnarray*}
showing that one of $x$ and $x_{new}$ preserves the pairing between $W_1$ and $W_2 \oplus W_3$ if and only if the other does. The isometry condition
\[
B_n(x_{11} u_1, x_{33} u_3) = B_n(u_1, u_3)
\]
also shows that $x_{11}$ must be invertible. Now, $x$ preserves the orthogonality of $W_1$ and $W_4$ if and only if
\[
0 = B_n(x u_1, x u_4) = B_n(x_{11} u_1, x_{34}u_4) + B_n(x_{41}u_1, x_{44}u_4).
\]
We observe that the change of variables $u_1 \to x_{11}^{-1}u_1$ gives
\[
B_n(u_1, x_{34} u_4) = - B_n(x_{41}x_{11}^{-1}u_1, x_{44} u_4),
\]
which in turn uniquely determines $x_{34}$ since $B_n$ gives a non-degenerate pairing between $W_1$ and $W_3$. On the other hand, $x_{new}$ automatically preserves the orthogonality of $W_1$ and $W_4$ since $x_{new}W_1 \subseteq W_1 \oplus W_2 \oplus W_3 = W_4^\perp$ and $x_{new} W_4  = W_4$. It remains to determine the conditions under which $x$ and $x_{new}$ preserve $B_n$ on $W_1$. Suppose $u_1, v_1 \in W_1$. Then,
\begin{eqnarray*}
B_n(x u_1, x v_1) &=& B_n(x_{11} u_1, x_{31} v_1) + B_n(x_{21} u_1, x_{21} v_1) + B_n(x_{31} u_1, x_{11} v_1) + B_n(x_{41}u_1, x_{41} v_1), \\
B_n(x_{new}u_1, x_{new}v_1) &=& B_n(x_{11} u_1, (x_{31}+J) v_1) + B_n(x_{21} u_1, x_{21} v_1) + B_n((x_{31}+J) u_1, x_{11} v_1).
\end{eqnarray*}
These two expressions will be equal if we can find $J$ such that
\[
B_n(x_{11}u_1, Jv_1) + B_n(Ju_1, x_{11}v_1) = B_n(x_{41}u_1, x_{41}v_1),
\]
which in turn will guarantee that $B_n(x u_1, x v_1) = B_n(u_1, v_1)$ if and only if $B_n(x_{new}u_1, x_{new}v_1) = B_n(u_1,v_1)$.
\newline \newline \noindent
Since $x_{41}$ could be arbitrary, all we know about $B_m(x_{41}u_1, x_{41}v_1)$ is that it is some possibly degenerate sesquilinear form on $W_1$. Performing the change of coordinates $u_1 \to x_{11}^{-1}u_1$, $v_1 \to x_{11}^{-1}v_1$, $J \to Jx_{11}$, we would have to find $J \in \hom(W_1, W_3)$ such that
\[
B_n(u_1, Jv_1) + B_n(Ju_1, v_1) = B_n(x_{41}x_{11}^{-1}u_1, x_{41}x_{11}^{-1}v_1).
\]
This can be achieved by Lemma \ref{random_sesquilinear_realisation_lemma} (and this $J$ may be viewed as a function of $x_{41}$ and $x_{11}$).
\end{proof}

\begin{theorem} \label{sesquilinear_conjugacy_theorem}
Two bounding pairs in $BP_n$, $(g, V)$ and $(g^\prime, V^\prime)$ are conjugate by $G_\infty(\mathbb{F}_q)$ if and only if they are conjugate by $G_{B_n}(V_n)$.
\end{theorem}
\begin{proof}
If we have $x \in G_\infty(\mathbb{F}_q)$ such that $g^\prime = xgx^{-1}$ and $V^\prime = xV$, $x$ must be a member of some finite $G_{B_m}(V_m)$ (we may assume $m \geq n$). In particular, it follows that $V \cap V_m$ is isometric to $V^\prime \cap V_m$ (both spaces being equipped with the restriction of $B_m$). As in Proposition \ref{sesquilinear_structure_prop}, these spaces decompose as $\ker(B_m|_{V}) \oplus U^{\oplus r} \oplus W$, where $U$ is a hyperbolic plane, and $W$ is the germ.
\newline \newline \noindent
First we show that there exists $y \in G_{B_n}(V_n)$ such that $V^\prime = yV$. By Proposition \ref{witt_lemma_proposition}, this is equivalent to $V \cap V_n$ and $V^\prime \cap V_n$ being isometric. But passing from $V \cap V_n$ and $V^\prime \cap V_n$ to $V \cap V_m$ and $V^\prime \cap V_m$ respectively amounts to adding $m-n$ hyperbolic planes. This implies that each of $V \cap V_n$ and $V^\prime \cap V_n$ is isometric to $\ker(B_m|_{V}) \oplus U^{\oplus (r-(m-n))} \oplus W$, and therefore they are isometric to each other. The upshot of this is that when considering whether $(g, V)$ and $(g^\prime, V^\prime)$ are conjugate by an element of $G_{B_n}(V_n)$, we may assume $V = V^\prime$, and therefore we may put both in standard shape with the same decomposition $V_m = W_1 \oplus W_2 \oplus W_3 \oplus W_4$. By construction, $V_n^\perp \cap V_m \subseteq W_4$ and $W_1 \oplus W_2 \oplus W_3$ is contained in $V_n$.
\newline \newline \noindent
By Lemma \ref{sesquilinear_extension_lemma}, $x$ must have the following block form with respect to the decomposition $V_m = W_1 \oplus W_2 \oplus W_3 \oplus W_4$:
\[
x = \begin{pmatrix}
x_{11} & 0 & 0 & 0\\
x_{21} & x_{22} & 0 & 0\\
x_{31} & x_{32} & x_{33} & x_{34} \\
x_{41} & 0 & 0 & x_{44}
\end{pmatrix}.
\]
The condition $xgx^{-1} = g^\prime$ may be rewritten as $x(g-\Id) = (g^\prime-\Id) x$ (subject to $x \in G_{B_m}(V_m)$). By Lemma \ref{sesquilinear_shape_lemma}, the equation $x(g-\Id) = (g^\prime-\Id) x$ expressed in terms of block matrices is the same as was computed in the proof of Proposition \ref{gl_inf_gl_n_conjugacy_prop}. In particular, the matrix blocks $x_{31}, x_{34}, x_{41}, x_{44}$ do not appear in this equation. This means we may replace $x_{34}$ and $x_{41}$ with zero and $x_{44}$ with the identity matrix and $x_{31}$ by $x_{31}+J$ for some matrix $J$ (determined by $x_{11}$ and $x_{41}$ as in Lemma \ref{sesquilinear_extension_lemma}), and the resulting matrix
\[
x_{new} = \begin{pmatrix}
x_{11} & 0 & 0 & 0\\
x_{21} & x_{22} & 0 & 0\\
x_{31} + J & x_{32} & x_{33} & 0 \\
0 & 0 & 0 & \Id
\end{pmatrix}
\]
still satisfies $x_{new}gx_{new}^{-1} = g^\prime$. Moreover because $W_4$ contains $V_n^\perp \cap V_m$, and the complement $W_1 \oplus W_2 \oplus W_3$ is contained in $V_n$, $x_{new}$ may be viewed as an element of $G_{B_n}(V_n)$ which obeys $x_{new} \cdot (g, V) = (g^\prime, V^\prime)$.
\end{proof}

\begin{proposition}
There is an associative multiplication on bounding pairs of the same type defined by
\[
(g, V) \times (g^\prime, V^\prime)  = (g g^\prime, V \cap V^\prime).
\]
This multiplication is $G_{\infty}(\mathbb{F}_q)$-equivariant, meaning that if $x \in G_\infty(\mathbb{F}_q)$, then
\[
(x \cdot(g, V)) \times (x \cdot (g^\prime, V^\prime)) = x \cdot ((g, V) \times (g^\prime, V^\prime)).
\]
Finally, for any bounding pair $(h, U)$ there are only finitely many pairs $(g, V), (g^\prime, V^\prime)$ such that
\[
(g, V) \times (g^\prime, V^\prime)  = (h, U).
\]
\end{proposition}
\begin{proof}
The arguments are the same as in Subsection \ref{bounding_triples_subsection}.
\end{proof}

\begin{definition}
Let $\mathcal{A}^\prime$ be the set if functions from the set of all bounding pairs to $\mathbb{Z}$. It is an abelian group with pointwise addition. We equip $\mathcal{A}^\prime$ with the following convolution product. If $f_1, f_2 \in \mathcal{A}^\prime$, then
\[
(f_1 * f_2) (g^{\prime \prime}, V^{\prime \prime}) = \sum_{(g, V) \times (g^\prime, V^\prime) = (g^{\prime \prime}, V^{\prime \prime})} f_1(g, V) f_2(g^\prime, V^\prime).
\]
\end{definition}
\noindent
As in the case of general linear groups, this is well defined because the sum is finite. There is an action of $G_\infty(\mathbb{F}_q)$ on $\mathcal{A}^\prime$ via
\[
(x \cdot f)(g, V) = f(x^{-1} \cdot (g, V)).
\]
Because the multiplication is equivariant for the action of $G_\infty(\mathbb{F}_q)$, the product of two $G_\infty(\mathbb{F}_q)$-invariant functions is again invariant.

\begin{definition}
Let $\mathcal{A}^{G_\infty}$ be the subspace of $\mathcal{A}^\prime$ consisting of elements that are invariant for the action of $G_\infty(\mathbb{F}_q)$ and are supported on finitely many $G_\infty(\mathbb{F}_q)$-orbits of bounding pairs.
\end{definition}

\begin{proposition}
We have that $\mathcal{A}^{G_\infty}$ is a subalgebra of $\mathcal{A}^\prime$.
\end{proposition}

\begin{proof}
We modify the argument from Proposition \ref{Ivanov_Kerov_subalg_proposition}. Suppose that $f_1$ and $f_2$ are the indicator functions of two $G_\infty(\mathbb{F}_q)$-orbits containing $(g, V)$ and $(g^\prime, V^\prime)$ respectively. Then $(g, V) \times (g^\prime, V^\prime) = (gg^\prime, V \cap V^\prime)$ is an element of $BP_m$ for some $m$. We write $V_m = W_1 \oplus W_2 \oplus W_3 \oplus W_4$ putting $(gg^\prime, V \cap V^\prime)$ in standard shape. Note that $\codim(V \cap V^\prime) = \dim(W_1) + \dim(W_2)$. Since $\dim(W_1) = \dim(W_3)$,
\[
\dim(V_m) - \dim(W_4) = \dim(W_1) + \dim(W_2) + \dim(W_3) \leq 2 \codim(V \cap V^\prime).
\]
Hence $\dim(W_4) \geq \dim(V_m) - 2 \codim(V \cap V^\prime)$.
\newline \newline \noindent
The form $B_m$ is non-degenerate when restricted to $W_4$, so $W_4 = W \oplus U^{\oplus r}$ where $W$ is the germ of $W_4$ and $U$ is a hyperbolic plane. By the classification of anisotropic spaces (Theorem \ref{anisotropic_space_classification_theorem}), $\dim(W) \leq 2$. By counting dimensions and substituting the previous inequality, 
\[
2r \geq \dim(W_4)-2 \geq \dim(V_m) - 2 - 2\codim(V \cap V^\prime).
\]
This means that $V \cap V^\prime$ contains $V_m^\perp$ in addition to at least $\frac{\dim(V_m) - 2 - 2\codim(V \cap V^\prime)}{2}$ hyperbolic planes. By Proposition \ref{witt_lemma_proposition}, we may apply an element of $G_\infty(\mathbb{F}_q)$ to the sum of $V_m$ and these hyperbolic planes to obtain $V_{d}$, where $d$ is as follows. We have
\[
\dim(V_d)  = \dim(V_m) - 2r \leq \dim(V_m) - (\dim(V_m) - 2 - 2\codim(V \cap V^\prime)) = 2 + 2\codim(V \cap V^\prime).
\]
Since $\dim(V_d)$ is either $d$, $2d$, or $2d+1$ (depending on the nature of the sesquilinear form we are working with) in any case we have
\[
d \leq 2 + 2 \codim(V) + 2 \codim(V^\prime).
\]
We conclude that the $G_\infty(\mathbb{F}_q)$-orbit of $(gg^\prime, V \cap V^\prime)$ intersects $BP_{2 + 2 \codim(V) + 2 \codim(V^\prime)}$, and there are finitely many $G_\infty(\mathbb{F}_q)$-orbits in any $BP_d$, so $f_1 * f_2$ is supported on finitely many orbits.
\end{proof}

\begin{proposition}\label{surj_hom_classical_groups}
For any $n \in \mathbb{Z}_{\geq 0}$ there is a surjective homomorphism $\Psi_n: \mathcal{A}^\prime \to \mathbb{Z}G_{B_n}(V_n)$ defined by
\[
\Psi_n(f) = \sum_{(g, V) \in BP_n} f(g,V)g.
\]
Moreover, the image of $\mathcal{A}^{G_\infty}$ is precisely the centre $Z(\mathbb{Z}G_{B_n}(V_n))$.
\end{proposition}

\begin{proof}
The arguments are the same as in Section \ref{specitalisation_section}.
\end{proof}

\begin{proposition} \label{specialisation_partial_computation_proposition}
Let $f$ be the indicator function of $G_\infty(\mathbb{F}_q)$ orbit of $(g,V)$, where $(g, V) \in BP_n$. Let $m \geq n$ and let $Cl(g)$ be the sum of all elements in $G_{B_m}(V_m)$ that are conjugate to $g$. Then 
\[
\Psi_m(f) = K \frac{|C_{G_{B_m}(V_m)}(g)|}{|C_{G_{B_n}(W_1 \oplus W_2 \oplus W_3)}(g)| |\hom(W_1,W_4)| |G_{B_m}(W_4)|} Cl(g),
\]
where $V_m = W_1 \oplus W_2 \oplus W_3 \oplus W_4$ puts $g$ in standard shape, and $K$ is a positive integer independent of $m$.
\end{proposition}

\begin{proof}
We mimic the proof of Proposition \ref{specialisation_function_gl_proposition}. Any $(g^\prime, V^\prime) \in BP_m$ that is in the orbit of $(g, V)$ is conjugate by an element of $G_{B_m}(V_m)$, so $g^\prime$ is in the same $G_{B_m}(V_m)$-conjugacy class as $g$. Since $\Psi_m(f)$ is in the centre of the group algebra, we have $\Psi_m(f) = P \cdot Cl(g)$, and it remains to determine the scalar $P$ as a function of $m$. This multiplicity is equal to the number of elements $(g, V^\prime) \in BP_m$ that are conjugate to $(g, V)$. Explicitly, for some $x \in G_{B_m}(V_m)$,
\begin{eqnarray*}
V^\prime &=& xV \\
g &=& xgx^{-1}.
\end{eqnarray*}
So $P$ is the size of the orbit of $V$ under the action of the centraliser $C_{G_{B_m}(V_m)}(g)$. By the orbit-stabiliser relation,
\[
P = \frac{|C_{G_{B_m}(V_m)}(g)|}{|C_{G_{B_m}(V_m)}(g) \cap \mathrm{Stab}_m(V)|},
\]
where $\mathrm{Stab}_m(V)$ is the subgroup of $G_{B_m}(V_m)$ stabilising the subspace $V$ of $V_\infty$. We choose a decomposition $V_m = W_1 \oplus W_2 \oplus W_3 \oplus W_4$ putting $(g,V)$ in standard shape so that the elements of $\mathrm{Stab}_m(V)$ necessarily have the block matrix form
\[
x = \begin{pmatrix}
x_{11} & 0 & 0 & 0\\
x_{21} & x_{22} & 0 & 0\\
x_{31} & x_{32} & x_{33} & x_{34} \\
x_{41} & 0 & 0 & x_{44}
\end{pmatrix},
\]
where the block sizes are $a,\codim(V)-a,a,m-\codim(V)-a$ (where $a = \dim(V\cap V^\perp) = \dim(W_3)$).
\newline \newline \noindent
Now we reverse the argument in the proof of Theorem \ref{sesquilinear_conjugacy_theorem} to go from an element
\[
x_{small} = \begin{pmatrix}
x_{11} & 0 & 0 & 0\\
x_{21} & x_{22} & 0 & 0\\
x_{31} & x_{32} & x_{33} & 0 \\
0 & 0 & 0 & 1
\end{pmatrix}
\in G_{B_n}(W_1 \oplus W_2 \oplus W_3) \subseteq G_{B_m}(V_m)
\]
that fixes $V$ and commutes with $g$ to an element
\[
x = \begin{pmatrix}
x_{11} & 0 & 0 & 0\\
x_{21} & x_{22} & 0 & 0\\
x_{31} - J & x_{32} & x_{33} & x_{34} \\
x_{41} & 0 & 0 & x_{44}
\end{pmatrix}
\in G_{B_m}(V_m)
\]
that also fixes $V$ and commutes with $g$. Using Lemma \ref{sesquilinear_extension_lemma}, we see that we see that each $x_{small}$ gives rise to $|\hom(W_1, W_4)||G_{B_m}(W_4)|$ matrices $x$ via this construction (corresponding to choices of the entries $x_{41}$ and $x_{44}$). Moreover, since $J$ may be taken to be a function of $x_{11}$ and $x_{41}$, $x_{small}$ is determined by $x$.
\newline \newline \noindent
Let $N$ be the number of matrices $x_{small}$ that satisfy $x_{small}V = V$ and commute with $g$. Then
\[
|C_{G_{B_m}(V_m)}(g) \cap \mathrm{Stab}_m(V)| = 
N |\hom(W_1, W_4)| |G_{B_m}(W_4)|.
\]
So we finally get
\[
P = \frac{|C_{G_{B_n}(W_1 \oplus W_2 \oplus W_3)}(g)|}{N} \frac{|C_{G_{B_m}(V_m)}(g)|}{|C_{G_{B_n}(W_1 \oplus W_2 \oplus W_3)}(g)| |\hom(W_1,W_4)| |G_{B_m}(W_4)|}.
\]
Here we let $K = \frac{|C_{G_{B_n}(W_1 \oplus W_2 \oplus W_3)}(g)|}{N}$, which is an integer because the set of matrices $x_{small}$ is a subgroup of $C_{G_{B_n}(W_1 \oplus W_2 \oplus W_3)}(g)$.
\end{proof}
\noindent
Now we are in a position to compute, $\Psi_m(f)$ with cases for the unitary, symplectic, and orthogonal groups. In the notation of Proposition \ref{specialisation_partial_computation_proposition}, we let $a = \dim(W_1) = \dim(W_3)$, so that $\dim(W_2) = \codim(V) - \dim(W_1)$ and $\dim(W_4) = m - \dim(W_1) - \dim(W_2) - \dim(W_3) = m - \codim(V) - a$. Suppose that $g \in G_{B_n}(W_1 \oplus W_2 \oplus W_3)$ has type $\bs\mu$. 

\begin{proposition} \label{unitary_specialisation_proposition}
In the setting of Proposition \ref{specialisation_partial_computation_proposition}, suppose we are working with Hermitian forms and unitary groups. Let $r = \dim(W_1 \oplus W_2 \oplus W_3) = \codim(V) + a$, and let $\bs\mu$ be the (unmodified) type of $g$ viewed as an element of $U_r(\mathbb{F}_q)$. Then
\[
\Psi_m(f) =  K q^{2(m - r)(k-\frac{h}{2}-a)}
\qbinom{m-r+h}{h}_{-q}
Cl(g),
\]
where $k = l(\bs\mu(t-1))$ and $h = m_1(\bs\mu(t-1))$.
\end{proposition}
\begin{proof}
When viewed as an element of $G_{B_m}(V_m)$, the type of $g$ is $\bs\mu \cup (1^{m-r})_{t-1}$. By Corollary \ref{unitary_centraliser_ratio_corollary}, we have
\begin{eqnarray*}
\Psi_m(f) &=& K \frac{|C_{G_{B_m}(V_m)}(g)|}{|C_{G_{B_n}(W_1 \oplus W_2 \oplus W_3)}(g)| |\hom(W_1,W_4)| |G_{B_m}(W_4)|} Cl(g) \\
&=& K q^{2(m-r)(k-h)} \frac{|U_{m-r+h}(\mathbb{F}_q)|}{|U_h(\mathbb{F}_q)|}\frac{Cl(g)}{|\hom(W_1,W_4)| |U_{m-r}(\mathbb{F}_q)|}.
\end{eqnarray*}
Now we use the fact that the ground field is $\mathbb{F}_{q^2}$, so $|\hom(W_1, W_4)| = q^{2\dim(W_1)\dim(W_4)} = q^{2a(m-r)}$. This gives us
\begin{eqnarray*}
\Psi_m(f) &=&  K q^{2(m - r)(k-h-a)}
\frac{ q^{m-r+h \choose 2}
\prod_{i=1}^{m-r+h} (q^i - (-1)^i)
}{
q^{m-r \choose 2}
\prod_{i=1}^{m-r} (q^i - (-1)^i)
\cdot
q^{h \choose 2}
\prod_{i=1}^{h} (q^i - (-1)^i)
}
Cl(g) \\
&=&K q^{2(m - r)(k-\frac{h}{2}-a)}
\frac{\prod_{i=m-r+1}^{m-r+h} ((-q)^i - 1)}{\prod_{i=1}^{h} ((-q)^i - 1)}
Cl(g) \\
&=& K q^{2(m - r)(k-\frac{h}{2}-a)}
\qbinom{m-r+h}{h}_{-q}
Cl(g).
\end{eqnarray*}
\end{proof}
\begin{remark}  \label{exponent_nonnegativity_remark}
Similarly to Lemma \ref{lower_case_specialisation_correctness_lemma}, the proposition remains correct for any $m \geq 0$ (not just $m \geq n$). Proposition \ref{exponent_nonnegativity_proposition} applies to $g$ (where $b = \dim(W_3) = \dim(W_1) = a$), showing that $2k - 2a - h \geq 0$, and hence $k - \frac{h}{2} - a \geq 0$. As a result, the quantity appearing in Proposition \ref{unitary_specialisation_proposition} may be viewed as a polynomial in $(-q)^m$. Although the minus sign may seem peculiar, it may be viewed as a facet of \emph{Ennola duality}, which asserts that that a wide range of quantities associated to $U_n(\mathbb{F}_q)$ can be obtained from those for $GL_n(\mathbb{F}_q)$ under the substitution $q \to -q$. See \cite{FRST} for more details on Gaussian binomial coefficients at $-q$, and \cite{ThiemVinroot} for a detailed discussion of Ennola duality.
\end{remark}

\begin{proposition} \label{symplectic_specialisation_proposition}
In the setting of Proposition \ref{specialisation_partial_computation_proposition}, suppose we are working with alternating forms and symplectic groups and $q$ is odd. Let $r = \dim(W_1 \oplus W_2 \oplus W_3) = \codim(V) + a$, and let $\bs\mu$ be the (unmodified) type of $g$ viewed as an element of $Sp_r(\mathbb{F}_q)$. Then
\[
\Psi_m(f) = K q^{(2m-r)(k-\frac{h}{2}-a)} \qbinom{m - \frac{r}{2} + \frac{h}{2}}{\frac{h}{2}}_{q^2} Cl(g).
\]
where $k = l(\bs\mu(t-1))$ and $h = m_1(\bs\mu(t-1))$.
\end{proposition}
\begin{proof}
When viewed as an element of $G_{B_m}(V_m)$, the type of $g$ is $\bs\mu \cup (1^{2m-r})_{t-1}$. By Corollary \ref{symplectic_centraliser_ratio_corollary}, 
\begin{eqnarray*}
\Psi_m(f) &=& K \frac{|C_{G_{B_m}(V_m)}(g)|}{|C_{G_{B_n}(W_1 \oplus W_2 \oplus W_3)}(g)| |\hom(W_1,W_4)| |G_{B_m}(W_4)|} Cl(g) \\
&=& K q^{(2m-r)(k-h)} \frac{|Sp_{2m-r+h}(\mathbb{F}_q)|}{|Sp_h(\mathbb{F}_q)|}\frac{Cl(g)}{|\hom(W_1,W_4)| |Sp_{2m-r}(\mathbb{F}_q)|}.
\end{eqnarray*}
Now we use the fact that the ground field is $\mathbb{F}_q$, so $|\hom(W_1, W_4)| = q^{a(2m-r)}$. This gives us
\begin{eqnarray*}
\Psi_m(f) &=& K q^{(2m-r)(k-h-a)}
\frac{q^{(\frac{2m-r+h}{2})^2}\prod_{i=1}^{\frac{2m-r+h}{2}}(q^{2i}-1)}{q^{(\frac{2m-r}{2})^2}\prod_{i=1}^{\frac{2m-r}{2}}(q^{2i}-1) \cdot q^{(\frac{h}{2})^2}\prod_{i=1}^{\frac{h}{2}}(q^{2i}-1)} Cl(g) \\
&=& K q^{(2m-r)(k-\frac{h}{2}-a)}
\frac{\prod_{i=\frac{2m-r}{2}+1}^{\frac{2m-r+h}{2}}(q^{2i}-1)}{\prod_{i=1}^{\frac{h}{2}}(q^{2i}-1)} Cl(g) \\
&=& K q^{(2m-r)(k-\frac{h}{2}-a)} \qbinom{m - \frac{r}{2} + \frac{h}{2}}{\frac{h}{2}}_{q^2}.
\end{eqnarray*}
\end{proof}
\begin{remark} 
Similarly to Lemma \ref{lower_case_specialisation_correctness_lemma}, the proposition remains correct for any $m \geq 0$ (not just $m \geq n$). As in Remark \ref{exponent_nonnegativity_remark}, we have $k - \frac{h}{2} - a \geq 0$. So the quantity appearing in Proposition \ref{symplectic_specialisation_proposition} may be viewed as a polynomial in $q^{2m}$. Since $r$ is necessarily even (as it is the dimension of a space with a non-degenerate symplectic form), we may take the coefficients to be in $\mathbb{Z}[q^2, q^{-2}]$.
\end{remark}

\begin{remark} \label{symplectic_char_2_remark_2}
By Remark \ref{symplectic_char_2_remark}, Corollary \ref{symplectic_centraliser_ratio_corollary} still applies in characteristic 2. So the calculation in Proposition \ref{symplectic_specialisation_proposition} also goes through in characteristic 2.
\end{remark}

\begin{proposition} \label{orthogonal_specialisation_proposition}
In the setting of Proposition \ref{specialisation_partial_computation_proposition}, suppose we are working with symmetric forms and orthogonal groups and $q$ is odd. Let $r = \dim(W_1 \oplus W_2 \oplus W_3) = \codim(V) + a$. Then for a certain $P(t) \in \mathcal{R}_q[\tfrac{1}{2}]$, we have $\Psi_m(f) = P(q^m) Cl(g)$.
\end{proposition}

\begin{proof}
Let $M = \dim(V_m)$, which is either $2m$ or $2m+1$ according to whether the ambient group is $O_{2m}^{\pm}(\mathbb{F}_q)$ or $O_{2m+1}(\mathbb{F}_q)$. Also let $\bs\mu$ be the (unmodified) type of $g$ viewed as an element of $O_r^{\pm}(\mathbb{F}_q)$, $k = l(\bs\mu(t-1))$, and $h = m_1(\bs\mu(t-1))$ as before. We have
\[
\Psi_m(f) =
K q^{(M-r)(k-h-a)} \frac{|O_{h+M-r}^{\epsilon_1 \oplus \epsilon_2}|}{|O_h^{\epsilon_1}||O_{M-r}^{\epsilon_2}|},
\]
and there are in principle 16 cases according to the four possible values of each of the two germs $\epsilon_1, \epsilon_2$ (which determine a third germ $\epsilon_3 = \epsilon_1 \oplus \epsilon_2$). For simplicity we group these cases according to the parities of $h$ and $M-r$ and omit the intermediate calculations (which are similar to the unitary and symplectic cases). Instead of writing $+$ or $-$ for the germs $\epsilon_i$, we write $+1$ or $-1$, so that the formula for $|O_{2n}^\epsilon|$ (given in Proposition \ref{classical_group_sizes_proposition}) contains a factor of $(q^n - \epsilon)$.
\newline \newline \noindent
Case 1: $h$ and $M-r$ both odd.
\[
\Psi_m(f) = \frac{K}{2} q^{\frac{M-r}{2}(2k-h-2a) - \frac{1}{2}} \qbinom{\frac{M-r-1}{2} + \frac{h-1}{2}}{\frac{h-1}{2}}_{q^2} (q^{\frac{M-r+h}{2}} - \epsilon_3).
\]
Note that in this case we have a denominator, $2$.
\newline \newline \noindent
Case 2: $h$ odd and $M-r$ even.
\[
\Psi_m(f) = \frac{K}{2} q^{\frac{M-r}{2}(2k-h-2a)} \qbinom{\frac{M-r}{2} + \frac{h-1}{2}}{\frac{h-1}{2}}_{q^2} (q^{\frac{M-r}{2}} + \epsilon_2).
\]
There is also a denominator in this case.
\newline \newline \noindent
Case 3: $h$ even and $M-r$ odd.
\[
\Psi_m(f) = \frac{K}{2} q^{\frac{M-r}{2}(2k-h-2a)} \qbinom{\frac{M-r-1}{2} + \frac{h}{2}}{\frac{h}{2}}_{q^2} (q^{\frac{h}{2}} + \epsilon_1).
\]
In this case the denominator $2$ divides $q^{\frac{h}{2}}+\epsilon_1$, which is a constant independent of $m$.
\newline \newline \noindent
Case 4: $h$ and $M-r$ both even. In this case we can write $\epsilon_3 = \epsilon_1 \epsilon_2$ (a multiplicative way of expressing $\mathbf{0} \oplus \omega = \omega$, $\omega \oplus \omega = \mathbf{0}$, etc. in the Witt ring).
\[
\Psi_m(f) = \frac{K}{2} q^{(\frac{M-r}{2})(2k-h-2a)}\qbinom{\frac{M-r}{2} + \frac{h}{2}-1}{\frac{h}{2}-1}_{q^2} \frac{(q^{\frac{M-r}{2}} + \epsilon_2)(q^{\frac{M-r+h}{2}}-\epsilon_3)}{(q^{\frac{h}{2}}-\epsilon_1)}.
\]
In this case, the denominator $(q^{\frac{h}{2}}-\epsilon_1)$ is constant with respect to $m$.
\newline \newline \noindent
Note that $\qbinom{m}{k}_{q^2}$ is a polynomial in $q^{2m}$, and therefore also a polynomial in $q^m$ (of twice the degree). Remark \ref{exponent_nonnegativity_remark} guarantees $(2k-h-2a) \geq 0$, so in any of the above cases, the result is a polynomial in the variable $q^m$.
\newline \newline \noindent
In Case 4 we have a polynomial in $q^m$ with rational coefficients. For notational convenience, let $u = \frac{m-r}{2}$ and $v = \frac{h}{2}$. To confirm that we actually have an element of $\mathcal{R}_q[\tfrac{1}{2}]$, we show that for any $u$ and $v$,
\[
\qbinom{u+v-1}{v-1}_{q^2} \frac{(q^{u} + \epsilon_2)(q^{u+v}-\epsilon_3)}{(q^{v}-\epsilon_1)} \in \mathbb{Z}[q,q^{-1}].
\]
First we point out that
\[
\frac{(q^{u} + \epsilon_2) (q^{u+v}-\epsilon_1\epsilon_2)}{(q^{v}-\epsilon_1)} 
= \epsilon_2 q^{u} + 1 + q^{v} \left( \frac{q^{2u}-1}{q^{v} - \epsilon_1}\right)
\]
So it is enough to show that the quantity $\qbinom{u+v-1}{v-1}_{q^2} \frac{q^{2u}-1}{q^{v} - \epsilon_1}$ is an element of $\mathbb{Z}[q,q^{-1}]$. To do this, we factor the expression in terms of cyclotomic polynomials $\Phi_d(q)$ which obey $q^n - 1 = \prod_{d|n} \Phi_d(q)$ and are themselves elements of $\mathbb{Z}[q]$. The multiplicity $\mathrm{Mult}_d$ of $\Phi_d(q)$ in
\[
\qbinom{u+v-1}{v-1}_{q^2} = 
\frac{\prod_{i=1}^{u+v-1} (q^{2i}-1)}{\prod_{j=1}^{u} (q^{2j}-1) \prod_{k=1}^{v-1} (q^{2k}-1)}
\]
is
\[
\mathrm{Mult}_d = \left\{
        \begin{array}{ll}
            \lfloor \frac{u+v-1}{d} \rfloor - \lfloor \frac{u}{d} \rfloor - \lfloor \frac{v-1}{d}
            \rfloor & \quad d \mbox{ odd} \\
            \lfloor \frac{u+v-1}{d/2} \rfloor - \lfloor \frac{u}{d/2} \rfloor - \lfloor \frac{v-1}{d/2}
            \rfloor & \quad d \mbox{ even}        \end{array}
    \right.
\]
Subcase 4.1: $\epsilon_2 = -1$. The denominator $q^v+1$ is the product of $\Phi_d(q)$ over $d$ dividing $2v$ but not $v$. Such $d$ are necessarily even, and $d/2$ divides $v$. Accordingly, $\mathrm{Mult}_d$ is $0$ if $d/2$ divides $u$, and $1$ otherwise. In the former case, $d$ divides $2u$ and hence $\Phi_d(q)$ is a factor of the numerator $q^{2u}-1$, and in the latter case, we may take the factor of $\Phi_d(q)$ from $\qbinom{u+v-1}{v-1}_{q^2}$.
\newline \newline \noindent
Subcase 4.2: $\epsilon_2 = 1$. The denominator $q^v-1$ is the product of $\Phi_d(q)$ over $d$ dividing $v$. If $d$ is odd and $d$ does not divide $2u$, then also $d$ does not divide $u$, and $\mathrm{Mult}_d = 1$. If $d$ is even and $d$ does not divide $2u$, $\mathrm{Mult}_d = 1$. So regardless of parity, if $d$ does not divide $2u$, we are done. If $d | 2u$, we may take the factor $\Phi_d(q)$ from the numerator $q^{2u}-1$.
\end{proof}

\section{Stable Centres for the Classical Groups}
\noindent
In this section, we state stability properties about the centres of group algebras of the classical groups, analogous to what was done in Section \ref{stable_centres_for_general_linear_groups}. We will handle each case separately. Let us retain the notation from Section \ref{classical_groups_section}. 

\subsection{Stable Centres for the Unitary Groups}
As discussed in the Appendix, for any $g \in U_n(\mathbb{F}_q)$, its type as an element in $g \in GL_n(\mathbb{F}_{q^2})$ determines its conjugacy class. Additionally, there is an involution $r \mapsto r^*$ on $\Phi_{q^2}$ such that the type $\bs\mu$ of $g \in U_n(\mathbb{F}_q)$ obeys $\bs\mu(r) = \bs\mu(r^*)$. Let us say that multipartitions that satisfy this operation are \textit{invariant under $*$}. Like with the general linear group, we use the notion of a modified type (c.f. Definition \ref{modified_type_definition} and subsequent discussion) in order to discuss conjugacy classes for all $n$ at once. In particular, if $g \in U_n(\mathbb{F}_q)$ has modified type $\bs\mu$, then so does $g$ when viewed as an element of $U_{n+1}(\mathbb{F}_q)$ under the prescribed inclusion. Recall that the type of $g$ can be recovered from the modified type if we are given $n$.

\begin{definition}
Let $X_{\bs\mu, n}$ denote the sum of all elements of modified type $\bs\mu$ in $U_n(\mathbb{F}_q)$, viewed as an element of $Z(\mathbb{Z}U_n(\mathbb{F}_q))$.
\end{definition}
\noindent
Once again, this is either the sum over a conjugacy class in $U_n(\mathbb{F}_q)$, or it is zero. We now state the analogue of Theorem \ref{gl_structure_constant_theorem} for the unitary groups.

\begin{theorem} \label{u_structure_constant_theorem}
There is a family of elements $r_{\bs\mu, \bs\nu}^{\bs\lambda} \in \mathcal{R}_{-q}$ interpolating the structure constants of $Z(\mathbb{Z}U_n(\mathbb{F}_q))$ in the following way:
\[
X_{\bs\mu,n} X_{\bs\nu,n} = \sum_{\bs\lambda} r_{\bs\mu, \bs\nu}^{\bs\lambda}([n]_{-q}) X_{\bs\lambda, n}.
\]
Here, $\bs\mu, \bs\nu$ are arbitrary multipartitions invariant under $*$ and the sum ranges over all multipartitions invariant under $*$.
\end{theorem}

\begin{proof}
The proof is almost identical to that of Theorem \ref{gl_structure_constant_theorem}, except instead of appealing to results in Section \ref{specitalisation_section} for the general linear group, we use the analogous results for the unitary group from Section \ref{classical_groups_section}.
\end{proof}
\noindent
As before, we can construct an algebra that interpolate the centers of the group algebras of the unitary groups:
\begin{definition}
Let $\mathrm{FH}_q^{U}$ be the free $\mathcal{R}_{-q}$-module with basis given by symbols $K_{\bs\mu}$ for multipartitions $\bs\mu$ on $\Phi_{q^2}$ invariant under $*$. We equip $\mathrm{FH}_q^{U}$ with a bilinear multiplication defined on basis elements via
\[
K_{\bs\mu} K_{\bs\nu} = \sum_{\bs\lambda} r_{\bs\mu, \bs\nu}^{\bs\lambda} K_{\bs\lambda},
\]
where $r_{\bs\mu, \bs\nu}^{\bs\lambda} \in \mathcal{R}_{-q}$ are the elements from Theorem \ref{u_structure_constant_theorem}. We call $\mathrm{FH}_q^{U}$ the \emph{unitary Farahat-Higman algebra}.
\end{definition}
\noindent
We have the following analogue of Corollary \ref{gl_special_hom_cor}:
\begin{corollary}\label{u_special_hom_cor}
There is a ``specialisation'' homomorphism $ \Theta_n: \mathrm{FH}_q^{U} \to Z(\mathbb{Z}U_n(\mathbb{F}_q))$ defined by $\Theta_n(K_{\bs\mu}) = X_{\bs\mu,n}$ and by evaluating the coefficients (elements of $\mathcal{R}_{-q}$) at $[n]_{-q}$, where $\bs\mu$ is any multipartition invariant under $*$.
\end{corollary}

\begin{proof}
This is a consequence of Theorem \ref{u_structure_constant_theorem}.
\end{proof}

\begin{proposition}\label{ufh_is_nice_prop}
We have that $\mathrm{FH}_q^{U}$ is an associative, commutative, unital $\mathcal{R}_{-q}$-algebra.
\end{proposition}
\begin{proof}
The proof of this proposition is analogous to the proof of Proposition \ref{glfh_is_nice_prop}.
\end{proof}

\begin{lemma} \label{u_degree_bound_lemma}
For any multipartitions $\bs\mu, \bs\nu, \bs\lambda$ invariant under $*$, the degree of the polynomial $r_{\bs\mu, \bs\nu}^{\bs\lambda}$ is at most $2(|\bs\mu| + |\bs\nu| - |\bs\lambda|)$.
\end{lemma}

\begin{proof}
The proof is very similar to that of Lemma \ref{degree_bound_lemma}, but there are some modifications that need to be made to account for the different standard form. Let $g, g' \in U_\infty(\mathbb{F}_q)$ be any elements of modified type $\bs\mu, \bs\nu$, respectively, and let $V, V'$ be their fixed point spaces, respectively. Let $f_{\bs\mu}$ and $f_{\bs\nu}$ be the indicator functions on the $U_\infty(\mathbb{F}_q)$ orbits of the tight bounding pairs $(g, V)$ and $(g', V')$, respectively.  Then, recalling the surjective homomorphism $\Psi_m: \mathcal{A}^{U_\infty} \rightarrow Z(\mathbb{Z}U_m(\mathbb{F}_q))$ from Proposition \ref{surj_hom_classical_groups}, we have the following by the proof of Theorem \ref{u_structure_constant_theorem}:

\[ \Psi_m(f_{\bs\mu}f_{\bs\nu}) = \Psi_m(f_{\bs\mu})\Psi_m(f_{\bs\nu}) =  X_{\bs\mu, m}X_{\bs\mu, m} = \sum_{\bs\lambda} r^{\bs\lambda}_{\bs\mu,\bs\nu}([m]_{-q}) X_{\bs\lambda,m}.\]
The argument of $\Psi_m$ in the leftmost term is an integral linear combination of various indicator functions on $U_\infty(\mathbb{F}_q)$ orbits of elements of the form $(s'', W'')$, where $(s'', W'') = (s, W)\times (s',W')$ and $(s,W), (s',W')$ are $U_\infty(\mathbb{F}_q)$-conjugate to $(g,V), (g',V')$, respectively. In particular, $(s,W)$ and $(s',W')$ will necessarily be tight and have modified type $\bs\mu$ and $\bs\nu$, respectively. Let us pick one such $(s'', W'')$, and suppose it has modified type $\bs\lambda$. Writing it in standard form with respect to $V_n$ for some suitable $V_n$, the image of the corresponding indicator function under $\Psi_m$ is 

\[P((-q)^m)X_{\bs\lambda, m} = K q^{2(m - r)(k-\frac{h}{2}-a)}\qbinom{m-r+h}{h}_{-q} X_{\bs\lambda,m}\]
by Proposition \ref{unitary_specialisation_proposition} (using the notation from that proposition). This means that $P$ is a polynomial of $(-q)^m$ with coefficients in $\mathbb{Z}[q, q^{-1}]$ of degree $h + 2(k - h/2 - a) = 2k - 2a$. On the other hand, by tightness, we have $\codim(W) = |\bs\mu|$ and $\codim(W') = |\bs\nu|$, so that $\codim(W'') \leq |\bs\mu| + |\bs\nu|$. But $\codim(W'') = r - a$ (see Proposition \ref{unitary_specialisation_proposition}). By the definition of $r$ and the fact that $\bs\lambda$ is the modified type of $s''$, we have $|\bs\lambda| = r - k$. Finally, putting this together, we have $2(|\bs\mu| + |\bs\nu|) \geq 2(r-a) = 2(r-k) + 2(k-a) = 2|\bs\lambda| + \deg P$, which gives the desired inequality for the polynomial $P$. Now, in order to get the contribution to $r_{\bs\mu, \bs\nu}^{\bs\lambda}$, we need to sum such polynomials $P$ over all possible choices of $(s, W)$ and $(s',W')$ such that $(s'', W'')$ is of modified type $\bs\lambda$. However, the bound only depends on $\bs\mu, \bs\nu, \bs\lambda$, so we deduce that degree of $r_{\bs\mu, \bs\nu}^{\bs\lambda}$ also satisfies the desired inequality.
\end{proof}
\noindent
This gives rise to the analogue of Proposition \ref{gl_assoc_graded} for the unitary groups. 
\begin{theorem}
The algebra $\mathrm{FH}_q^{U}$ is filtered, where $K_{\bs\mu}$ is in filtration degree $|\bs\mu|$. Moreover the structure constants of the associated graded algebra are integers.
\end{theorem}

\begin{proof}
This follows from Lemma \ref{u_degree_bound_lemma} (see the proof of Proposition \ref{gl_assoc_graded}).
\end{proof}
\noindent 
This addresses one of the further directions recommended in \cite{Wan_Wang}, and also in \cite{Meliot_Hecke}.

\subsection{Stable Centres for the Symplectic Groups}
For the symplectic group $Sp_{2n}(\mathbb{F}_q)$, two elements are conjugate only if their types match. However, this is not a sufficient condition; as we note in the appendix, conujugacy classes in $Sp_{2n}(\mathbb{F}_q)$ are indexed by certain multipartitions $\bs\mu$ of size $2n$ with some additional signs making $\bs\mu(r \pm 1)$ into symplectic signed partitions.
\newline \newline \noindent
Let $g \in Sp_{2n}(\mathbb{F}_q)$ be an element of the conjugacy class labelled by the signed multipartition $\bs\mu$. We define the \emph{symplectic modified type} of $g$ to be the signed multipartition obtained by subtracting $1$ from each part of $\bs\mu(t-1)$. Note that this turns $\bs\mu(t-1)$ from a symplectic signed partition into an orthogonal signed partition. The symplectic modified type of $g$ is the same when $g$ is viewed as an element of $Sp_{2(n+1)}(\mathbb{F}_q)$. The conjugacy class of $g$ can be recovered from the symplectic modified type of $g$ if $n$ is known.

\begin{definition}
Let $X_{\bs\mu, n}$ denote the sum of all elements of symplectic modified type $\bs\mu$ in $Sp_{2n}(\mathbb{F}_q)$, viewed as an element of $Z(\mathbb{Z}Sp_{2n}(\mathbb{F}_q))$.
\end{definition}
\noindent
This is either the sum over a conjugacy class in $Sp_{2n}(\mathbb{F}_q)$, or it is zero. We now state the analogue of Theorem \ref{gl_structure_constant_theorem} for the symplectic groups.

\begin{theorem} \label{sp_structure_constant_theorem}
There is a family of elements $r_{\bs\mu, \bs\nu}^{\bs\lambda} \in \mathcal{R}_{q^2}$ interpolating the structure constants of $Z(\mathbb{Z}Sp_{2n}(\mathbb{F}_q))$ in the following way:
\[
X_{\bs\mu,n} X_{\bs\nu,n} = \sum_{\bs\lambda} r_{\bs\mu, \bs\nu}^{\bs\lambda}([n]_{q^2}) X_{\bs\lambda, n}.
\]
Here, $\bs\mu, \bs\nu$ are arbitrary signed multipartitions corresponding to a symplectic modified type and the sum ranges over all signed multipartitions corresponding to a symplectic modified type.
\end{theorem}

\begin{proof}
The proof is almost identical to that of Theorem \ref{gl_structure_constant_theorem}, except instead of appealing to results in Section \ref{specitalisation_section} for the general linear group, we use the analogous results for the symplectic group from Section \ref{classical_groups_section}.
\end{proof}
\noindent
As before, we can construct an algebra that interpolate the centers of the group algebras of the symplectic groups:
\begin{definition}
Let $\mathrm{FH}_q^{Sp}$ be the free $\mathcal{R}_{q^2}$-module with basis given by symbols $K_{\bs\mu}$ for signed multipartitions $\bs\mu$ on $\Phi_q$ that correspond to a symplectic modified type. We equip $\mathrm{FH}_q^{Sp}$ with a bilinear multiplication defined on basis elements via
\[
K_{\bs\mu} K_{\bs\nu} = \sum_{\bs\lambda} r_{\bs\mu, \bs\nu}^{\bs\lambda} K_{\bs\lambda},
\]
where $r_{\bs\mu, \bs\nu}^{\bs\lambda} \in \mathcal{R}_{q^2}$ are the elements from Theorem \ref{sp_structure_constant_theorem}. We call $\mathrm{FH}_q^{Sp}$ the \emph{symplectic Farahat-Higman algebra}.
\end{definition}
\noindent
We have the following analogue of Corollary \ref{gl_special_hom_cor}:
\begin{corollary}\label{sp_special_hom_cor}
There is a ``specialisation'' homomorphism $ \Theta_n: \mathrm{FH}_q^{Sp} \to Z(\mathbb{Z}Sp_{2n}(\mathbb{F}_q))$ defined by $\Theta_n(K_{\bs\mu}) = X_{\bs\mu,n}$ and by evaluating the coefficients (elements of $\mathcal{R}_{q^2}$) at $[n]_{q^2}$, where $\mu$ is any signed multipartition corresponding to a symplectic modified type.
\end{corollary}

\begin{proof}
This is a consequence of Theorem \ref{sp_structure_constant_theorem}.
\end{proof}

\begin{proposition}\label{spfh_is_nice_prop}
We have that $\mathrm{FH}_q^{Sp}$ is an associative, commutative, unital $\mathcal{R}_{q^2}$-algebra.
\end{proposition}
\begin{proof}
The proof of this proposition is analogous to the proof of Proposition \ref{glfh_is_nice_prop}.
\end{proof}

\begin{lemma} \label{sp_degree_bound_lemma}
For any signed multipartitions $\bs\mu, \bs\nu, \bs\lambda$ corresponding to a symplectic modified type, the degree of the polynomial $r_{\bs\mu, \bs\nu}^{\bs\lambda}$ is at most $|\bs\mu| + |\bs\nu| - |\bs\lambda|$.
\end{lemma}

\begin{proof}
The proof is almost identical to that of Lemma \ref{u_degree_bound_lemma}, except we lose a factor of $2$ because the polynomials lie in $\mathcal{R}_{q^2}$ and not $\mathcal{R}_{-q}$.
\end{proof}
\noindent
This gives rise to the analogue of Proposition \ref{gl_assoc_graded} for the symplectic groups. 
\begin{proposition}[Theorem 4.30, \cite{OZDEN2021263}]
The algebra $\mathrm{FH}_q^{Sp}$ is filtered, where $K_{\bs\mu}$ is in filtration degree $|\bs\mu|$. Moreover the structure constants of the associated graded algebra are integers.
\end{proposition}

\begin{proof}
This follows from Lemma \ref{sp_degree_bound_lemma} (see the proof of Proposition \ref{gl_assoc_graded}).
\end{proof}

\begin{remark}
By Remark \ref{symplectic_char_2_remark_2}, it is also possible to construct a version of $\mathrm{FH}_q^{Sp}$ when $q$ is a power of 2. However, the indexing set of conjugacy classes is somewhat more complicated, so we forgo the details.
\end{remark}

\subsection{Stable Centres for the Orthogonal Groups} 
Let $G_n(\mathbb{F}_q)$ be one of the the three families of groups $O_{2n+1}(\mathbb{F}_q)$, $O_{2n}^+(\mathbb{F}_q)$, $O_{2n}^-(\mathbb{F}_q)$, so that passing from $n$ to $n+1$ does not change the germ of the natural representation of the group. 
\newline \newline \noindent
Similarly to the symplectic case, conjugacy classes in orthogonal groups are indexed by certain multipartitions $\bs\mu$ with additional signs making $\bs\mu(t \pm 1)$ orthogonal signed partitions. If $\bs\mu$ is the signed multipartition describing a the conjugacy class of an element $g$ of an orthogonal groups, we define the \emph{orthogonal modified type} of $g$ to be the signed multipartition obtained by subtracting 1 from each part of $\bs\mu(t-1)$. This turns $\bs\mu(t-1)$ from an orthogonal signed partition into a symplectic signed partition, with one complication. If $1$ appears as a part in $\bs\mu(t-1)$, it has an associated sign, and subtracting 1 results in this sign being associated to parts of size zero (which is not part of the data associated to a symplectic signed partition). Nevertheless, the orthogonal modified type of $g$ is unchanged by the prescribed inclusions of orthogonal groups (provided that in the absence of a sign associated to parts of size $0$ in $\bs\mu(t-1)$, we take the sign to be $+$; this is because the germ of a zero dimensional space is zero). Again, the conjugacy class of $g$ can be recovered from the orthogonal modified type if $n$ is known.

\begin{definition}
Let $X_{\bs\mu, n}$ denote the sum of all elements of orthogonal modified type $\bs\mu$ in $G_{n}(\mathbb{F}_q)$, viewed as an element of $Z(\mathbb{Z}G_{n}(\mathbb{F}_q))$.
\end{definition}
\noindent
Once again, this is either the sum over a conjugacy class in $G_{n}(\mathbb{F}_q)$, or it is zero. We now state the analogue of Theorem \ref{gl_structure_constant_theorem} for the orthogonal groups.

\begin{theorem} \label{o_structure_constant_theorem}
There is a family of elements $r_{\bs\mu, \bs\nu}^{\bs\lambda} \in \mathcal{R}_{q}[\frac{1}{2}]$ interpolating the structure constants of $Z(\mathbb{Z}G_{n}(\mathbb{F}_q))$ in the following way:
\[
X_{\bs\mu,m} X_{\bs\nu, m} = \sum_{\bs\lambda} r_{\bs\mu, \bs\nu}^{\bs\lambda}([m]_{q}) X_{\bs\lambda, m}.
\]
Here, $\bs\mu, \bs\nu$ are arbitrary multipartitions corresponding to orthogonal modified types and the sum ranges over all multipartitions corresponding to orthogonal modified types.
\end{theorem}

\begin{proof}
The proof is almost identical to that of Theorem \ref{gl_structure_constant_theorem}, except instead of appealing to results in Section \ref{specitalisation_section} for the general linear group, we use the analogous results for the orthogonal group from Section \ref{classical_groups_section}.
\end{proof}
\noindent
As before, we can construct an algebra that interpolate the centers of the group algebras of the orthogonal groups:
\begin{definition}
Let $\mathrm{FH}_q^{O,+}, \mathrm{FH}_q^{O,-}, \mathrm{FH}_q^{O,odd}$ be the free $\mathcal{R}_{q}[\frac{1}{2}]$-modules with basis given by symbols $K_{\bs\mu}$ indexed by orthogonal modified types $\bs\mu$ coming from groups of the form $O_{2n}^+(\mathbb{F}_q), O_{2n}^-(\mathbb{F}_q), O_{2n+1}(\mathbb{F}_q)$, respectively. We equip each $\mathrm{FH}_q^{O,*}$ with a bilinear multiplication defined on basis elements via
\[
K_{\bs\mu} K_{\bs\nu} = \sum_{\bs\lambda} r_{\bs\mu, \bs\nu}^{\bs\lambda} K_{\bs\lambda},
\]
where $r_{\bs\mu, \bs\nu}^{\bs\lambda} \in \mathcal{R}_{q}[\frac{1}{2}]$ are the elements from Theorem \ref{o_structure_constant_theorem}. We call $\mathrm{FH}_q^{O,+},\mathrm{FH}_q^{O,-},\mathrm{FH}_q^{O,odd}$ the \emph{positive orthogonal Farahat-Higman algebra}, the \emph{negative orthogonal Farahat-Higman algebra}, and the \emph{odd orthogonal Farahat-Higman algebra}, respectively.
\end{definition}
\noindent
We have the following analogue of Corollary \ref{gl_special_hom_cor}:
\begin{corollary}\label{o_special_hom_cor}
There are ``specialisation'' homomorphisms $ \Theta_n: \mathrm{FH}_q^{O,*} \to Z(\mathbb{Z}G_{n}(\mathbb{F}_q))$ (here $* = +,-,odd$ according to whether $G_{n} = O_{2n}^+, O_{2n}^-, O_{2n+1}$) defined by $\Theta_n(K_{\bs\mu}) = X_{\bs\mu,n}$ and by evaluating the coefficients (elements of $\mathcal{R}_{q}[\frac{1}{2}]$) at $[n]_{q}$.
\end{corollary}

\begin{proof}
This is a consequence of Theorem \ref{o_structure_constant_theorem}.
\end{proof}

\begin{proposition}\label{ofh_is_nice_prop}
We have that each $\mathrm{FH}_q^{O,*}$ is an associative, commutative, unital $\mathcal{R}_{q}[\frac{1}{2}]$-algebra.
\end{proposition}
\begin{proof}
The proof of this proposition is analogous to the proof of Proposition \ref{glfh_is_nice_prop}.
\end{proof}

\begin{lemma} \label{o_degree_bound_lemma}
The degree of the polynomial $r_{\bs\mu, \bs\nu}^{\bs\lambda}$ is at most $2(|\bs\mu| + |\bs\nu| - |\bs\lambda|)$.
\end{lemma}

\begin{proof}
The proof is analogous to that of Lemma \ref{u_degree_bound_lemma}, although there are four cases as in the proof of Proposition \ref{orthogonal_specialisation_proposition}.
\end{proof}

\begin{theorem}
The algebras $\mathrm{FH}_q^{O, *}$ are filtered, where $K_{\bs\mu}$ is in filtration degree $|\bs\mu|$. Moreover the structure constants of the associated graded algebra are integers.
\end{theorem}

\begin{proof}
In versions of this statement for the other classical groups, the proof was similar to that of proof of Proposition \ref{gl_assoc_graded}. In this case, if we apply similar logic, we can deduce that the first part of the statement that the structure constants are rational (because the polynomials $r_{\bs\mu,\bs\nu}^{\bs\lambda}$ lie in $\mathcal{R}_{q^2}[\frac{1}{2}]$). However, on the other hand, these structure constants are manifestly integral because they are counting the multiplicity of a certain conjugacy class sum appearing in the the expansion in the product of two other conjugacy class sums. 
\end{proof}
\noindent 
This addresses one of the further directions recommended in \cite{Wan_Wang}.

\printbibliography

\end{document}